\documentclass[11pt]{article}

\usepackage[margin=1in]{geometry}
\usepackage[T1]{fontenc}
\usepackage[utf8]{inputenc}
\usepackage{lmodern}
\usepackage{microtype}
\usepackage{amsmath,amssymb,amsthm,mathtools}
\usepackage{enumitem}
\usepackage{aliascnt}
\usepackage[colorlinks=true,linkcolor=blue,citecolor=blue,urlcolor=blue]{hyperref}
\hypersetup{
  pdftitle={Sparse chromatic graphs and the complete-graph triangle bound},
  pdfauthor={Shuyan Chen},
  pdfsubject={Extremal graph theory and graph coloring},
  pdfkeywords={chromatic number, triangles, critical graph, list coloring, sparse graph}
}
\usepackage[nameinlink,capitalise,noabbrev]{cleveref}

\newtheorem{theorem}{Theorem}[section]

\newaliascnt{lemma}{theorem}
\newtheorem{lemma}[lemma]{Lemma}
\aliascntresetthe{lemma}

\newaliascnt{proposition}{theorem}
\newtheorem{proposition}[proposition]{Proposition}
\aliascntresetthe{proposition}

\newaliascnt{corollary}{theorem}
\newtheorem{corollary}[corollary]{Corollary}
\aliascntresetthe{corollary}

\theoremstyle{remark}
\newaliascnt{remark}{theorem}

\aliascntresetthe{remark}

\newcommand{\tri}{\mathrm t}
\newcommand{\ex}{\operatorname{ex}}
\newcommand{\Log}{\operatorname{Log}}
\newcommand{\LogLog}{\operatorname{LogLog}}
\newcommand{\cP}{\mathcal P}

\title{Sparse chromatic graphs and the complete-graph triangle bound}
\author{Shuyan Chen\\
\small Department of Mathematics, University of Manchester\\
\small \texttt{shuyan.chen-2@student.manchester.ac.uk}}
\date{}

\begin{document}
\maketitle

\begin{abstract}
We prove that there is an absolute constant $c>0$ such that every graph of chromatic number at least $r$ and at most $cr^3\log^2 r$ edges contains at least $\binom r3$ triangles.

The proof has three ingredients.  First, a sparse-core argument based on a triangle-sensitive coloring estimate of Harris extracts, from any counterexample, an induced subgraph of order $O(r)$ and chromatic number at least $(1-\eta)r$.  Second, we prove an order-uniform stability theorem: for every $\beta>0$ there is $\gamma>0$, independent of the constant in the linear order bound, such that every sufficiently large $s$-critical graph $J$ of order $O(s)$ with $\omega(J)\le (1-\beta)s$ has at least $\binom s3+\gamma s^3$ triangles.  The proof combines the excess method and the modified-independent-set argument of Fox, Tidor, and Zhang.  Third, we prove the exact bound when the graph contains a clique of order at least $(1-\delta)r$.  This uses a new dense common-palette list analogue of Harris's edge--triangle estimate: if every list occupies a fixed positive proportion of a common palette, then the edge--triangle coloring bound survives up to a constant factor.

\end{abstract}

\section{Introduction}

All graphs are finite and simple, and all logarithms not carrying the truncated symbol $\Log$ are natural.  For a graph $G$, write $\chi(G)$, $e(G)$, $\tri(G)$, and $\omega(G)$ for its chromatic number, number of edges, number of triangles, and clique number.  We call $G$ $s$-critical if $\chi(G)=s$ and every proper induced subgraph has chromatic number less than $s$.  Whenever a critical subgraph is chosen below, it is understood to be induced in the ambient graph under discussion.  Fox, Tidor, and Zhang proved that every sufficiently large $r$-chromatic graph in which every edge lies in sufficiently many triangles contains at least $\binom r3$ triangles \cite{FTZ}.

The global sparsity scale is sharp up to constant factors: the minimum number of edges in a triangle-free graph of chromatic number at least $r$ is $\Theta(r^3\log^2 r)$ \cite{Kim,Nilli,PoljakTuza}; see also the discussion in \cite{FTZ}.  The difficulty is the exact target.  Coloring estimates naturally distinguish $0$ from $\Theta(r^3)$ triangles, whereas the desired statement must exclude even $\binom r3-1$.

Our main result is the following.

\begin{theorem}\label{thm:main}
There is an absolute constant $c>0$ such that, for every integer $r\ge3$, every graph $G$ satisfying
\[
 \chi(G)\ge r,
 \qquad
 e(G)\le cr^3\log^2 r
\]
satisfies
\[
 \tri(G)\ge\binom r3.
\]
\end{theorem}

This settles Conjecture~6.1 of Fox, Tidor, and Zhang \cite{FTZ}.

We first prove the statement for all sufficiently large $r$; decreasing $c$ deals with the finitely many remaining values.

The proof is assembled from three statements.  The first is a sparse chromatic-core theorem.

\begin{theorem}[Sparse core]\label{thm:core}
For every $\eta>0$ there are constants $c_\eta,L_\eta>0$ such that, for all sufficiently large $r$, the following holds.  If
\[
 \chi(G)\ge r,
 \qquad
 e(G)\le c_\eta r^3\log^2r,
 \qquad
 \tri(G)<\binom r3,
\]
then $G$ has an induced subgraph $H$ satisfying
\[
 \chi(H)\ge(1-\eta)r,
 \qquad
 |H|\le L_\eta r.
\]
Moreover, $H$ may be chosen with bounded independence number, where the bound depends only on $\eta$.
\end{theorem}

The second ingredient says that, away from a nearly maximum clique, a linear-order critical graph has a fixed cubic surplus of triangles.  The important feature is that the surplus does not depend on the constant $C$ in $|J|\le Cs$.

\begin{theorem}[Order-uniform triangle stability]\label{thm:uniformStability}
For every $\beta>0$ there is $\gamma=\gamma(\beta)>0$ such that the following holds.  For every fixed $C>0$ and all sufficiently large $s$ in terms of $C$ and $\beta$, every $s$-critical graph $J$ satisfying
\[
 |J|\le Cs,
 \qquad
 \omega(J)\le(1-\beta)s
\]
has
\[
 \tri(J)\ge\binom s3+\gamma s^3.
\]
\end{theorem}

The third ingredient treats the complementary near-clique regime exactly.

\begin{theorem}[Near-clique case]\label{thm:nearClique}
There are absolute constants $c_0,\delta_0>0$ such that, for all sufficiently large $r$, every graph $G$ satisfying
\[
 \chi(G)\ge r,
 \qquad
 e(G)\le c_0r^3\log^2r,
 \qquad
 \omega(G)\ge(1-\delta_0)r
\]
has at least $\binom r3$ triangles.
\end{theorem}

To prove \cref{thm:nearClique}, we precolor a large clique and obtain lists on the remaining vertices.  These lists are dense subsets of one common palette.  The following theorem is a dense common-palette list analogue of Harris's edge--triangle estimate \cite{Harris}.  Only the instance $(\rho,A)=(1/3,4)$ is used in the proof of the main theorem; the general form records the parameter inheritance needed by the proof.

For every real $x\ge0$, put
\[
 \Log x=\max\{1,\log(\max\{x,1\})\},
 \qquad
 \LogLog x=\Log(\Log x).
\]
Thus $\Log 0=\LogLog 0=1$.  Whenever a nonnegative local-triangle
parameter $y$ occurs in a denominator, we write
\[
 y_+=\max\{1,y\}.
\]
With these conventions the estimates below also cover $m=0$, $t=0$, and
$y=0$ without changing their nontrivial ranges.  Degree and triangle thresholds may be real; whenever an integer layer index is required, the floor or ceiling is displayed explicitly.

\begin{theorem}[Dense common-palette coloring]\label{thm:densePalette}
For every $\rho,A>0$ there are constants $K=K(\rho,A)$ and $p_0=p_0(\rho,A)$ such that the following holds whenever $p\ge p_0$.  Let $F$ be a graph with at most $p^A$ nonisolated vertices, $m$ edges, and $t$ triangles.  Let $L(v)$ be lists from a common palette $\cP$ of $p$ colors, with
\[
 |L(v)|\ge\rho p
 \qquad(v\in V(F)).
\]
If
\[
 K\left(
   \frac{m^{1/3}}{\Log^{2/3}m}+t^{1/3}
 \right)\le p,
\]
then $F$ is $L$-colorable.
\end{theorem}

The organization is as follows.  In \cref{sec:tools} we record the results of Fox--Tidor--Zhang, Harris, and Vu used below.  The sparse core is proved in \cref{sec:core}.  The full dense common-palette proof and the near-clique theorem occupy \cref{sec:palette,sec:near}.  In \cref{sec:stability} we prove first a fixed-order stability statement by refining the excess argument of Fox--Tidor--Zhang, and then remove the dependence of the surplus on the order constant using their modified-weight deletion.  The proof of \cref{thm:main} is given in \cref{sec:assembly}.

\section{Tools from coloring theory}\label{sec:tools}

We use four inputs.  The first two are exact triangle results for critical graphs of linear order.

\begin{theorem}[Fox--Tidor--Zhang]\label{thm:FTZlinear}
For every fixed $C>0$ and all sufficiently large $s$, every $s$-critical graph $J$ with $|J|\le Cs$ satisfies
\[
 \tri(J)\ge\binom s3.
\]
\end{theorem}

\begin{proof}
If $|J|\le2s-2$, this is Corollary~3.2 of \cite{FTZ}.  If $2s-1\le|J|\le Cs$, this is Proposition~4.2 of \cite{FTZ}.
\end{proof}

The next lemma is the locally sparse degeneracy estimate proved by Fox, Tidor, and Zhang, following Molloy and Reed.

\begin{lemma}[Fox--Tidor--Zhang]\label{lem:degenerate}
For every $0<\varepsilon,\delta<1$ there is $\zeta>0$ such that the following holds for all sufficiently large $d$.  Let $F$ be a $d$-degenerate graph with a degeneracy order $v_1,\dots,v_n$, and write
\[
 N^+(v_i)=N(v_i)\cap\{v_1,\dots,v_{i-1}\}.
\]
Suppose that
\begin{enumerate}[label=\textup{(\roman*)}]
\item $|N^+(v_i)|\le d$ for every $i$, and $|N^+(v_i)|\le(1-\varepsilon)d$ for $i\le n-\delta d/4$;
\item $e(F[N^+(v)])\le(1-\delta)\binom d2$ for every $v$.
\end{enumerate}
Then
\[
 \chi(F)\le(1-\zeta)d.
\]
\end{lemma}

This is Lemma~3.5 of \cite{FTZ}.

We also need two convenient zero-parameter forms of triangle-sensitive
coloring estimates of Harris.  A graph has local triangle bound $y$ if
every vertex lies in at most $y$ triangles.  Harris works with the same
truncated-logarithm convention in the nonzero range.  The statements below
separate the harmless zero cases and absorb bounded parameter ranges into
one absolute constant $K_H$.

\begin{corollary}[Convenient form of Harris's vertex estimate]
\label{thm:HarrisVertex}
Every graph with $n$ vertices, $t$ triangles, and local triangle bound $y$
satisfies
\[
 \chi(G)\le K_H\left(
 \sqrt{\frac n{\Log n}}
 +t^{1/3}\frac{\LogLog(t^2/y_+^3)}{\Log^{2/3}(t^2/y_+^3)}
 \right).
\]
\end{corollary}

\begin{proof}
For $t>0$, this is Harris's Theorem~2.3 \cite{Harris}, under his
truncated-logarithm convention, after changing the absolute implicit
constant.  If $t=0$, the second term vanishes and the first term is the
triangle-free vertex estimate quoted as Theorem~1.1 in \cite{Harris}.
The empty graph is immediate.
\end{proof}

\begin{corollary}[Convenient form of Harris's edge estimate]
\label{thm:HarrisEdge}
Every graph with $m$ edges, $t$ triangles, and local triangle bound $y$
satisfies
\[
 \chi(G)\le 1+K_H\left(
 \frac{m^{1/3}}{\Log^{2/3}m}
 +t^{1/3}\frac{\LogLog(t^2/y_+^3)}{\Log^{2/3}(t^2/y_+^3)}
 \right).
\]
\end{corollary}

\begin{proof}
If $m=0$, then $G$ is edgeless and $\chi(G)\le1$.  For $m\ge1$ and
$t>0$, the assertion follows from Harris's Theorem~3.2 \cite{Harris}; for
$t=0$, use the triangle-free edge estimate quoted as Theorem~1.1 there.
The additive $1$ serves only to include the edgeless case and is absorbed
in every asymptotic application below.
\end{proof}

We now distinguish Vu's original asymptotic theorem from the truncated
version used in the proof.

\begin{theorem}[Vu]\label{thm:VuOriginal}
There is an absolute constant $C_V$ such that the following holds.  Let
$G$ have maximum degree $\Delta$, let $1<f<\Delta^2$, and suppose
\[
 e(G[N(v)])\le\frac{\Delta^2}{f}
 \qquad(v\in V(G)).
\]
Then
\[
 \chi_\ell(G)\le C_V\frac{\Delta}{\log f}.
\]
\end{theorem}

This is Vu's locally sparse list-coloring theorem \cite{Vu}.

\begin{corollary}[Truncated locally sparse list bound]\label{thm:Vu}
There is an absolute constant $K_V$ such that the following holds.  Let
$d\ge1$, and let $G$ have maximum degree at most $d$ and local triangle
bound $y$.  Then $G$ is colorable from every list assignment of size at
least
\[
 K_V\frac d{\Log(d^2/y_+)}.
\]
\end{corollary}

\begin{proof}
Let $\Delta=\Delta(G)$ and put $a=\sqrt{y_+}$.  If $\Delta=0$, one color
suffices.  For $x>0$, write
\[
 \frac{x}{\Log(x^2/y_+)}
 =a\,h(x/a),
 \qquad
 h(u)=\frac{u}{\Log(u^2)}.
\]
The function $h$ is nondecreasing for $u$ outside one fixed compact
interval and has a positive absolute lower bound on $u\ge1$.

If $\Delta^2/y_+$ is larger than a sufficiently large absolute constant,
apply \cref{thm:VuOriginal} with
$f=\Delta^2/(2y_+)$.  Then
$e(G[N(v)])\le y\le\Delta^2/f$, and
$\log f\ge c\Log(\Delta^2/y_+)$.  The normalized arguments are $\Delta/a$ and $d/a$.  The present case
places $\Delta/a$ beyond a fixed absolute threshold, and
$d/a\ge\Delta/a$, so both lie in the monotone range of $h$.  Hence
\[
 \frac{\Delta}{\Log(\Delta^2/y_+)}
 \le C\frac d{\Log(d^2/y_+)}.
\]
If $\Delta^2/y_+$ is bounded, then either $d<a$, in which case the
truncated denominator equals $1$, or $d\ge a$, in which case the positive
lower bound for $h$ gives
\[
 \frac d{\Log(d^2/y_+)}\ge ca\ge c'(\Delta+1).
\]
Thus greedy list coloring applies after increasing the absolute constant.
The two cases prove the corollary.
\end{proof}

The next lemma combines a bounded-independence core obtained from Harris's estimate with the Ramsey--Tur\'an calculation in the proof of Proposition~5.3 of \cite{FTZ}.  The preliminary core reduction allows the order bound $e(F)\le Ms^2$ for arbitrary fixed $M$, while the Ramsey--Tur\'an calculation yields an absolute cubic surplus.

\begin{lemma}[Small-clique surplus]\label{lem:smallCliqueSurplus}
There are absolute constants $\varepsilon_0,\sigma_0>0$ such that the following holds.  For every fixed $M>0$ and all sufficiently large $s$ in terms of $M$, every graph $F$ with
\[
 \chi(F)\ge s,
 \qquad
 e(F)\le Ms^2,
 \qquad
 \omega(F)\le\varepsilon_0s
\]
satisfies
\[
 \tri(F)\ge\binom s3+\sigma_0s^3.
\]
\end{lemma}

\begin{proof}
We first construct the bounded-independence core and then apply the Ramsey--Tur\'an part of the proof of Proposition~5.3 in \cite{FTZ}.  The constant $\sigma_0$ will be absolute, while the lower threshold for $s$ may depend on $M$.

Assume first that
\begin{equation}\label{eq:smallCliqueNear}
 \tri(F)\le2\binom s3.
\end{equation}
Pass to an induced $s$-critical subgraph, still denoted by $F$.  Since $\delta(F)\ge s-1$,
\begin{equation}\label{eq:smallCliqueOrder}
 |F|\le\frac{2e(F)}{s-1}\le3Ms
\end{equation}
for all large $s$.

We claim that $F$ contains a subgraph $F_0$ with
\begin{equation}\label{eq:smallCliqueCore}
 \chi(F_0)\ge\frac{9s}{10}
 \qquad\text{and}\qquad
 \alpha(F_0)\le K_0,
\end{equation}
where $K_0$ is an absolute constant.  To see this, choose a sufficiently large absolute constant $L$, put $y=s^2/L$, and partition $V(F)=A\sqcup B$ according as a vertex lies in at most $y$ triangles or more than $y$ triangles.  By \eqref{eq:smallCliqueNear},
\[
 |B|y<3\tri(F)\le6\binom s3<s^3,
 \qquad\text{so}\qquad
 |B|<Ls.
\]
Apply \cref{thm:HarrisVertex}, that is, Harris's Theorem~2.3, to $F[A]$.  The vertex term is
\[
 O_M\!\left(\sqrt{\frac{s}{\log s}}\right)=o_M(s)
\]
by \eqref{eq:smallCliqueOrder}.  If $\tri(F[A])=xs^3$, then $0\le x\le1/3$, and the local-triangle term divided by $s$ is
\[
 x^{1/3}\frac{\LogLog(x^2L^3)}{\Log^{2/3}(x^2L^3)}.
\]
The uniform estimate in \cref{lem:uniformPsi} with $X=1/3$ shows that this is arbitrarily small when $L$ is sufficiently large.  Thus $\chi(F[A])\le s/20$ for all sufficiently large $s$ in terms of $M$.  Hence $\chi(F[B])\ge19s/20$.

Repeatedly delete from $F[B]$ an independent set of size greater than a sufficiently large absolute integer $K_0>20L$.  Since $|B|<Ls$, fewer than $s/20$ such sets are deleted.  The remaining induced subgraph satisfies \eqref{eq:smallCliqueCore}.

We now choose the Ramsey--Tur\'an parameter explicitly.  Let
\[
 \mathcal K=\{2,4,\dots,2\lceil K_0/2\rceil\},
 \qquad
 \delta_{\mathrm{RT}}=\frac1{20K_0}.
\]
For each $k\in\mathcal K$, apply the Erd\H{o}s--S\'os theorem
\cite{ErdosSos} with parameters $k$ and $\delta_{\mathrm{RT}}$; denote
its clique-density constant by $\eta_k>0$.  Put
\[
 \eta_* = \min_{k\in\mathcal K}\eta_k,
 \qquad
 \varepsilon_0=\min\left\{\frac1{100},\frac{\eta_*}{20}\right\}.
\]
Only the finite set $\mathcal K$, and hence only the absolute number
$K_0$, enters this choice.

Put $s'=\chi(F_0)\ge9s/10$.  Starting with $G_0=F_0$, choose an induced
$(s'-i)$-critical subgraph $G_i'\subseteq G_i$, let $I_i$ be a maximum
independent set in $G_i'$, write
\[
 k_i=|I_i|=\alpha(G_i')\le K_0,
\]
and set $G_{i+1}=G_i'-I_i$.  Then $\chi(G_{i+1})=s'-i-1$ and
$\delta(G_i')\ge s'-i-1$.

Fix $i<4s/5$ and $v\in V(G_i')$.  For all sufficiently large $s$,
\[
 d_{G_i'}(v)\ge s'-i-1\ge s/11.
\]
Moreover $k_i\ne1$, since otherwise $G_i'$ would be a clique of order
$s'-i\ge s/11>\varepsilon_0s$, contrary to the clique hypothesis.
Let
\[
 k_i'=2\left\lceil\frac{k_i}{2}\right\rceil\in\mathcal K.
\]
Since $k_i\ge2$, we have $k_i'\le4k_i/3$.  The neighborhood graph
$H=G_i'[N(v)]$ satisfies
\[
 \alpha(H)\le k_i\le k_i'
 \quad\text{and}\quad
 \omega(H)\le\omega(F)\le\varepsilon_0s
 \le 11\varepsilon_0|H|\le\eta_{k_i'}|H|.
\]
The Erd\H{o}s--S\'os estimate therefore gives
\begin{align}
 e(G_i'[N(v)])
 &\ge\left(\frac1{k_i'}-\frac1{20K_0}\right)d_{G_i'}(v)^2
      -\frac12d_{G_i'}(v)\notag\\
 &\ge \frac7{10k_i}(s'-i-1)^2
      -\frac12d_{G_i'}(v).\label{eq:smallCliqueRTstep}
\end{align}
Here the coefficient $7/10$ follows from
\[
 \frac1{k_i'}\ge\frac3{4k_i},
 \qquad
 \frac1{20K_0}\le\frac1{20k_i}.
\]
For $i\ge4s/5$, ordinary Tur\'an gives the weaker estimate
\begin{equation}
 e(G_i'[N(v)])
 \ge \frac1{2k_i}(s'-i-1)^2
      -\frac12d_{G_i'}(v).\label{eq:smallCliqueTuranStep}
\end{equation}

Since $I_i$ is independent, each triangle counted at stage $i$ is
counted exactly once at that stage.  After $I_i$ is deleted, it cannot be
counted at a later stage.  Thus the triangles appearing in the stage sums
are pairwise disjoint across stages; triangles lost when passing to a
critical subgraph are simply not charged.  The edge sets incident with the
successive $I_i$ are also disjoint.  Summing
\eqref{eq:smallCliqueRTstep} over $v\in I_i$ for $i<4s/5$, summing
\eqref{eq:smallCliqueTuranStep} afterwards, and using $|I_i|=k_i$ gives
\[
 \tri(F)+\frac12e(F)
 \ge
 \frac7{10}\sum_{0\le i<4s/5}(s'-i-1)^2
 +\frac12\sum_{4s/5\le i<s'}(s'-i-1)^2.
\]
For each fixed admissible index $i$, the summand $(s'-i-1)^2$ is
nondecreasing in $s'$, and increasing $s'$ only adds further nonnegative
terms to the second sum.  Hence the right-hand side is minimized at the
smallest permitted value $s'=9s/10$.  Changing the integer endpoints by
$O(1)$ changes the sums by only $O(s^2)$, and
\begin{align*}
 \sum_{0\le i<4s/5}(9s/10-i-1)^2
   &=\frac{91}{375}s^3+O(s^2),\\
 \sum_{4s/5\le i<9s/10}(9s/10-i-1)^2
   &=\frac1{3000}s^3+O(s^2).
\end{align*}
Consequently the displayed lower bound equals
\[
 \left(\frac7{10}\cdot\frac{91}{375}
       +\frac12\cdot\frac1{3000}\right)s^3+O(s^2)
 =\frac{5101}{30000}s^3+O(s^2).
\]
Since
\[
 \binom s3=\frac{5000}{30000}s^3+O(s^2),
\]
we obtain, for all sufficiently large $s$,
\begin{equation}\label{eq:smallCliqueGain}
 \tri(F)+\frac12e(F)>1.02\binom s3.
\end{equation}
In particular,
\[
 \frac{5101}{30000}-\frac16=\frac{101}{30000}>0.
\]
Thus the actual leading cubic margin is $101s^3/30000$.  Since
$e(F)\le Ms^2$, the quadratic loss is absorbed once $s$ is sufficiently
large in terms of $M$; for example, the absolute choice
$\sigma_0=10^{-3}$ then gives
\[
 \tri(F)\ge\binom s3+\sigma_0s^3.
\]
Thus $\varepsilon_0$ and $\sigma_0$ are absolute; only the lower threshold
for $s$ depends on the fixed order constant $M$.

Finally, if \eqref{eq:smallCliqueNear} fails, the desired conclusion is immediate after decreasing $\sigma_0$.  This proves the lemma.
\end{proof}

\section{The sparse chromatic core}\label{sec:core}

We prove \cref{thm:core}.  The argument is an edge-count version of the bounded-independence extraction in \cite{FTZ}.

Define
\[
 \psi(z)=\frac{\LogLog z}{\Log^{2/3}z}.
\]

\begin{lemma}\label{lem:uniformPsi}
For every fixed $X>0$,
\[
 \sup_{0\le x\le X}x^{1/3}\psi(x^2L^3)\longrightarrow0
 \qquad\text{as }L\longrightarrow\infty.
\]
\end{lemma}

\begin{proof}
The function $\psi$ is bounded and tends to zero at infinity.  If $x\le L^{-1}$, then
\[
 x^{1/3}\psi(x^2L^3)\le L^{-1/3}\sup_{z>0}\psi(z).
\]
If $x>L^{-1}$, then $x^2L^3>L$.  Once $L$ is sufficiently large, $\psi$ is decreasing on $[L,\infty)$, so the expression is at most $X^{1/3}\psi(L)$.  Both bounds tend to zero.
\end{proof}

\begin{proof}[Proof of \cref{thm:core}]
Fix $\eta>0$.  Let $L$ be a sufficiently large constant and put
\[
 y=\frac{r^2}{L}.
\]
Partition $V(G)=A\sqcup B$, where $A$ consists of the vertices lying in at most $y$ triangles.  Since
\[
 \sum_{v\in V(G)}\tri_G(v)=3\tri(G)<3\binom r3<\frac{r^3}{2},
\]
we have
\begin{equation}\label{eq:Bsize}
 |B|<\frac{Lr}{2}.
\end{equation}

Apply \cref{thm:HarrisEdge} to $G[A]$.  Write $m_A=e(G[A])$ and $T_A=\tri(G[A])$.  The edge term is
\[
 \frac{m_A^{1/3}}{\Log^{2/3}m_A}
 \le O(c_\eta^{1/3}r)+o(r).
\]
Indeed, if $m_A\le r^2$ the term is $O(r^{2/3})$; otherwise $\Log m_A\ge2\log r$ and the claimed estimate follows from $m_A\le c_\eta r^3\log^2r$.

Write $x=T_A/r^3$.  The triangle term in \cref{thm:HarrisEdge}, divided by $r$, is
\[
 x^{1/3}\psi(x^2L^3).
\]
First choose $L$ so large that \cref{lem:uniformPsi} makes this term at most $\eta/(8K_H)$.  Then choose $c_\eta$ sufficiently small and $r$ sufficiently large so that the edge term contributes at most $\eta r/(8K_H)$ and the additive $1$ in \cref{thm:HarrisEdge} is at most $\eta r/4$.  We obtain
\[
 \chi(G[A])\le\frac{\eta r}{2},
\]
and therefore
\begin{equation}\label{eq:Bchi}
 \chi(G[B])\ge\left(1-\frac\eta2\right)r.
\end{equation}

Starting from $G[B]$, repeatedly delete an independent set of size greater than a fixed integer $K>L/\eta$, until the remaining induced subgraph $H$ has $\alpha(H)\le K$.  By \eqref{eq:Bsize}, fewer than $\eta r/2$ sets are deleted.  Each costs at most one color, so \eqref{eq:Bchi} yields
\[
 \chi(H)\ge(1-\eta)r.
\]
Also $|H|\le|B|<Lr/2$.  Renaming the order constant proves the theorem.
\end{proof}

\section{Dense common-palette coloring}\label{sec:palette}

Harris supplies the ordinary-coloring architecture and its numerical
optimizations, while Vu supplies the bounded-degree list-coloring input.  We
prove common-palette list analogues using palette reservation, global color
retirement, and recursive parameter inheritance.  The logical dependence of
the estimates is
\[
 \text{reservation and Vu}
 \Longrightarrow
 \{\text{vertex-local},\ \text{layered}\}
 \Longrightarrow
 \{\text{vertex-triangle},\ \text{edge-local}\}
 \Longrightarrow
 \text{edge-triangle}.
\]
Palette reservation supplies disjoint blocks for successive phases, and a
color used during a deletion phase is retired within its block.

\subsection{Palette bookkeeping and a common available color}

\begin{lemma}[Palette reservation]\label{lem:reservation}
For every $\rho,A>0$ there is $p_{\mathrm{res}}=p_{\mathrm{res}}(\rho,A)$ such that the following holds whenever $p\ge p_{\mathrm{res}}$.  Let $n\le p^A$, and
suppose $L(v)\subseteq\cP$, where $|\cP|=p$ and $|L(v)|\ge\rho p$ for
$n$ vertices.  If $1\le q\le\Log p$ is an integer, then $\cP$ has a
partition
\[
 \cP=\cP_1\sqcup\cdots\sqcup\cP_q
\]
such that, for every $v$ and $i$,
\begin{equation}\label{eq:reservationBounds}
 \frac{p}{2q}\le |\cP_i|\le\frac{2p}{q},
 \qquad
 |L(v)\cap\cP_i|\ge\frac{\rho p}{2q}.
\end{equation}
Consequently every restricted list has density at least $\rho/4$ in
$\cP_i$.  Moreover, for every $i$,
\begin{equation}\label{eq:inheritedExponent}
 n\le |\cP_i|^{2A}.
\end{equation}
\end{lemma}

\begin{proof}
Assign each color independently and uniformly to one of the $q$ parts.
Chernoff bounds show that any one of the inequalities in
\eqref{eq:reservationBounds} fails with probability
$\exp(-\Omega_\rho(p/q))$.  There are at most $(n+2)q\le(p^A+2)\Log p$
relevant events, while $p/q\ge p/\Log p$, so a union bound is valid for
all sufficiently large $p$.

Finally, $|\cP_i|\ge p/(2\Log p)\ge p^{1/2}$ for all sufficiently large
$p$.  Hence $p^A\le |\cP_i|^{2A}$, which proves
\eqref{eq:inheritedExponent}.
\end{proof}

We shall repeatedly use the following elementary logarithmic monotonicity,
which is the truncated-log version of \cite[Observation~1.6]{Harris}.

\begin{lemma}[Logarithmic monotonicity]\label{lem:logMonotonicity}
Let $0\le a\le1$ and $0<x_0\le x\le x_1\le z$.  Then
\[
 x_0\Log^a(z/x_0)
 \le x\Log^a(z/x)
 \le x_1\Log^a(z/x_1).
\]
\end{lemma}

\begin{proof}
On the range where $z/x>e$, the derivative of
$x(\log(z/x))^a$ is
$(\log(z/x))^{a-1}(\log(z/x)-a)\ge0$.  On the remaining range the
truncated logarithm equals $1$, so the function is simply $x$.
\end{proof}

For nonnegative arguments define the five control functions
\begin{align*}
 \Phi_V(u)&=\sqrt{\frac{u}{\Log u}},
 &\Phi_E(u)&=\frac{u^{1/3}}{\Log^{2/3}u},\\
 \Phi_U(u,v)&=\frac{u^{1/3}v_+^{1/3}}
                    {\Log^{2/3}(u^2/v_+)},
 &\Phi_W(u,v)&=\frac{u^{1/4}v_+^{1/4}}
                    {\Log^{3/4}(u/v_+)},
\end{align*}
and
\[
 \Phi_T(u,v)=
 \begin{cases}
  0,&u=0,\\[2mm]
  \displaystyle
  u^{1/3}\frac{\LogLog(u^2/v_+^3)}
                   {\Log^{2/3}(u^2/v_+^3)},&u>0.
 \end{cases}
\]
Thus the displayed bounds in the dense estimates below are sums of
members of this fixed family.

\begin{lemma}[Control-function monotonicity]\label{lem:controlMonotonicity}
For every $B\ge1$ there is $C_B\ge1$ with the following property.  If
$0\le u'\le Bu$ and $v'_+\le Bv_+$, then
\begin{align*}
 \Phi_V(u')&\le C_B\Phi_V(u),
 &\Phi_E(u')&\le C_B\Phi_E(u),\\
 \Phi_U(u',v')&\le C_B\Phi_U(u,v),
 &\Phi_W(u',v')&\le C_B\Phi_W(u,v),\\
 \Phi_T(u',v')&\le C_B\Phi_T(u,v).
\end{align*}
When $u=0$, the corresponding assertion is interpreted as $u'=0$.
\end{lemma}

\begin{proof}
Write $L(x)=\Log x$ and $L_2(x)=\LogLog x$.  For $x\ge e^3$ all
truncated logarithms below agree with ordinary logarithms.  Put
$s=\log(x^2)$.  Direct differentiation gives
\begin{align}
 x\frac{d}{dx}\log\left(\frac{x^{1/3}}{L(x^2)^{2/3}}\right)
 &=\frac13-\frac{4}{3s},\notag\\
 x\frac{d}{dx}\log\left(\frac{x^{1/4}}{L(x)^{3/4}}\right)
 &=\frac14-\frac{3}{4\log x},\notag\\
 x\frac{d}{dx}\log\left(
 x^{1/3}\frac{L_2(x^2)}{L(x^2)^{2/3}}\right)
 &=\frac13+\frac{2}{s\log s}-\frac{4}{3s},\label{eq:controlDerivatives}\\
 x\frac{d}{dx}\log\left(
 \frac{L_2(x^2)}{L(x^2)^{2/3}}\right)
 &=\frac{2}{s\log s}-\frac{4}{3s}.\notag
\end{align}
Since $s\ge6$ and $\log s>3/2$, the first three functions in
\eqref{eq:controlDerivatives} are nondecreasing on $[e^3,\infty)$, while
the last is nonincreasing there.  Similarly,
$x^{1/2}/L(x)^{1/2}$ and $x^{1/3}/L(x)^{2/3}$ are nondecreasing on this
range.  On the fully truncated ranges these functions reduce to powers of
$x$; the finitely many transition intervals are compact.  Hence there is an
absolute constant $C_0$ such that, whenever $0<x\le x'$, one has
\begin{align}
 \frac{x^{1/3}}{L(x^2)^{2/3}}
 &\le C_0\frac{x'^{1/3}}{L(x'^2)^{2/3}},
 &
 \frac{x^{1/4}}{L(x)^{3/4}}
 &\le C_0\frac{x'^{1/4}}{L(x')^{3/4}},\label{eq:quasiIncreasingUW}\\
 x^{1/3}\frac{L_2(x^2)}{L(x^2)^{2/3}}
 &\le C_0x'^{1/3}\frac{L_2(x'^2)}{L(x'^2)^{2/3}},
 &
 \frac{L_2(x'^2)}{L(x'^2)^{2/3}}
 &\le C_0\frac{L_2(x^2)}{L(x^2)^{2/3}}.\label{eq:quasiIncreasingT}
\end{align}
The same quasi-increase holds for
$x^{1/2}/L(x)^{1/2}$ and $x^{1/3}/L(x)^{2/3}$.  We shall also use the
exact monotonicities
\begin{equation}\label{eq:normalizedDecreasing}
 x\longmapsto
 \frac{x^{-2/3}}{L(x^2)^{2/3}}
 \quad\text{and}\quad
 x\longmapsto
 \frac{x^{-1/4}}{L(x)^{3/4}},
\end{equation}
which are nonincreasing on $(0,\infty)$.

The one-variable conclusions for $\Phi_V$ and $\Phi_E$ now follow at
once from $u'\le Bu$: first compare $u'$ with $Bu$, and then compare
$Bu$ with $u$ by a fixed dilation.

For the two-parameter functions, put $w=v_+$ and $w'=v'_+$, so
$w'\le Bw$.  Assume $u>0$, since otherwise $u'=0$ and there is nothing
to prove.  For $\Phi_U$, set
\[
 x=\frac{u}{w^{1/2}},
 \qquad
 x'=\frac{u'}{w'^{1/2}},
 \qquad
 h_U(x)=\frac{x^{1/3}}{L(x^2)^{2/3}}.
\]
Then $\Phi_U(u,v)=w^{1/2}h_U(x)$.  If $x'\le x$, the first inequality
in \eqref{eq:quasiIncreasingUW} and $w'\le Bw$ give
$\Phi_U(u',v')\le C_0B^{1/2}\Phi_U(u,v)$.  If $x'>x$, then
$w'^{1/2}x'=u'\le Bu=Bw^{1/2}x$, and hence
\[
 \Phi_U(u',v')
 \le Bw^{1/2}x\,\frac{h_U(x')}{x'}
 \le Bw^{1/2}x\,\frac{h_U(x)}x
 =B\Phi_U(u,v),
\]
where the second inequality is the first monotonicity in
\eqref{eq:normalizedDecreasing}.

For $\Phi_W$, put $x=u/w$, $x'=u'/w'$, and
$h_W(x)=x^{1/4}/L(x)^{3/4}$.  If $x'\le x$, use the second inequality
in \eqref{eq:quasiIncreasingUW} and $w'\le Bw$.  If $x'>x$, the relation
$w'x'=u'\le Bwx$ gives
\[
 \Phi_W(u',v')
 \le B^{1/2}w^{1/2}x^{1/2}
       \frac{h_W(x')}{x'^{1/2}}
 \le B^{1/2}w^{1/2}x^{1/2}
       \frac{h_W(x)}{x^{1/2}}
 =B^{1/2}\Phi_W(u,v),
\]
by the second monotonicity in \eqref{eq:normalizedDecreasing}.

Finally, for $\Phi_T$, put $x=u/w^{3/2}$, $x'=u'/w'^{3/2}$ and
\[
 h_T(x)=x^{1/3}\frac{L_2(x^2)}{L(x^2)^{2/3}}.
\]
Thus $\Phi_T(u,v)=w^{1/2}h_T(x)$.  If $x'\le x$, use the first
inequality in \eqref{eq:quasiIncreasingT} and $w'\le Bw$.  If $x'>x$,
then $w'^{3/2}x'=u'\le B w^{3/2}x$, and therefore
\begin{align*}
 \Phi_T(u',v')
 &\le B^{1/3}w^{1/2}x^{1/3}
       \frac{L_2(x'^2)}{L(x'^2)^{2/3}}\\
 &\le C_0B^{1/3}w^{1/2}x^{1/3}
       \frac{L_2(x^2)}{L(x^2)^{2/3}}
 =C_0B^{1/3}\Phi_T(u,v),
\end{align*}
where the second line is the quasi-decrease in
\eqref{eq:quasiIncreasingT}.  Taking the largest of the finitely many
constants proves the lemma.
\end{proof}

\begin{lemma}[Recursive palette-parameter inheritance]
\label{lem:parameterInheritance}
Fix $\rho,A>0$, an integer $q_0\ge1$, and $B\ge1$.  Suppose a palette
of size $p$ is split by \cref{lem:reservation} into $q\le q_0$ pieces,
and let $P'$ be one piece.  Suppose a recursive subproblem has at most
$n'\le n\le p^A$ relevant vertices and control quantity $R'\le BR$.
Let $K_*$ and $p_{0,*}$ be, respectively, the leading constant and the
lower palette threshold in the recursive theorem with parameters
$(\rho/4,2A+c)$, where $c\ge0$ is fixed.  If
\begin{equation}\label{eq:inheritanceParentCondition}
 p\ge 2q_0\max\{BK_*R,p_{0,*}\},
\end{equation}
then
\[
 |P'|\ge\frac{p}{2q_0},
 \qquad
 \text{list density in $P'$ at least }\frac\rho4,
 \qquad
 n'\le |P'|^{2A}\le |P'|^{2A+c},
\]
and, in addition,
\[
 |P'|\ge K_*R',
 \qquad
 |P'|\ge p_{0,*}.
\]

After any fixed number $h$ of nested reservations, the inherited density
and exponent may be taken to be $\rho/4^h$ and $2^hA+O_h(1)$.
\end{lemma}

\begin{proof}
The first three assertions are exactly
\eqref{eq:reservationBounds}--\eqref{eq:inheritedExponent}.  Under
\eqref{eq:inheritanceParentCondition},
\[
 |P'|\ge\frac p{2q_0}
 \ge\max\{BK_*R,p_{0,*}\}
 \ge\max\{K_*R',p_{0,*}\}.
\]
Enlarging a polynomial exponent only weakens the vertex-bound hypothesis.
Iteration proves the final assertion.
\end{proof}

The estimates below are proved in the order
\[
 \mathrm V,\ \mathrm L,\ \mathrm{VT},\ \mathrm{EL},\ \mathrm{ET}.
\]
Thus, when a parent estimate is proved, the constants in every child
estimate have already been fixed.  The inheritance lemma transfers those
child hypotheses to the reserved palettes; the resulting calls are
summarized in \cref{tab:recursiveCalls} after the five estimates.

\begin{lemma}[A common available color]\label{lem:weightedIndependent}
Let $R$ be a graph on a vertex set $S$ with at most $y$ edges.  Let
$M(u)$ be subsets of a palette $P$ such that
\[
 |M(u)|\ge\sigma|P|
 \qquad(u\in S).
\]
Then there are a color $c\in P$ and an independent set $I\subseteq S$
such that $c\in M(u)$ for every $u\in I$ and
\[
 |I|\ge\frac{\sigma^2|S|^2}{2y+|S|}.
\]
\end{lemma}

\begin{proof}
For $c\in P$, let $S_c=\{u\in S:c\in M(u)\}$ and put
$n_c=|S_c|$, $m_c=e(R[S_c])$.  Tur\'an's theorem gives an independent
set in $S_c$ of size at least
\[
 \frac{n_c^2}{2m_c+n_c}.
\]
By Cauchy--Schwarz,
\[
 \sum_{c\in P}\frac{n_c^2}{2m_c+n_c}
 \ge
 \frac{(\sum_cn_c)^2}{2\sum_cm_c+\sum_cn_c}.
\]
Now $\sum_cn_c\ge\sigma|P||S|$, while
$\sum_cm_c\le|P|y$ and $\sum_cn_c\le|P||S|$.  Dividing the resulting
lower bound by $|P|$ proves the claim.
\end{proof}

\subsection{The vertex and layered estimates}

\begin{lemma}[Dense vertex-local estimate]\label{lem:denseVertex}
For every $\rho,A>0$ there are constants $K=K(\rho,A)$ and $p_0=p_0(\rho,A)$ such that the following holds whenever $p\ge p_0$.  Let $F$ have $n\le p^A$ vertices
and local triangle bound $y$, and suppose $|L(v)|\ge\rho p$ from a
common palette of $p$ colors.  If
\[
 K\left(
 \sqrt{\frac n{\Log n}}
 +\frac{n^{1/3}y_+^{1/3}}{\Log^{2/3}(n^2/y_+)}
 \right)\le p,
\]
then $F$ is $L$-colorable.
\end{lemma}

\begin{proof}
We prove a common-palette list analogue of Harris's Theorem~2.1 \cite{Harris}, following its deletion architecture.  Bounded values of $n$ are absorbed by increasing $p_0(\rho,A)$, so assume $n$ is large.  We may
replace $y_+$ by $Y=\min\{y_+,\binom n2\}$.  This remains a local
triangle bound, and \cref{lem:controlMonotonicity} gives
$\Phi_U(n,Y)\le C\Phi_U(n,y)$, so it is enough to prove the estimate
with $Y$.  Reserve two palette blocks $P_D,P_R$ by \cref{lem:reservation}.  Then
$|P_D|,|P_R|\ge p/4$, and every restricted list has size at least
$\rho p/4$ and hence density at least $\rho/4$ in its block.

Fix a degree threshold $d\ge1$.  In the current graph, while some vertex
$v$ has degree greater than $d$, put $S=N(v)$.  The graph $F[S]$ has at
most $Y$ edges.  Suppose that at most $(\rho/8)|P_D|$ colors have already
been retired from $P_D$.  The current deletion list of every vertex then
has density at least $\rho/8$ in $P_D$.  By
\cref{lem:weightedIndependent}, there is an independent set in $S$
sharing one unretired color and having size at least
\[
 c_\rho\frac{d^2}{Y+d}.
\]
Color that set, delete it, and retire the color.  Since the deleted sets
are disjoint, the number $R$ of retired colors is at most
\begin{equation}\label{eq:denseVertexRetired}
 R\le C_\rho\left(\frac{nY}{d^2}+\frac nd\right).
\end{equation}
Since $|P_D|\ge p/4$, the deletion lists retain density at least $\rho/8$ whenever
\[
 p\ge\frac{32C_\rho}{\rho}
 \left(\frac{nY}{d^2}+\frac nd\right).
\]
When deletion stops, the residual graph has maximum degree at most $d$.
Its untouched lists in $P_R$ have size at least $\rho p/4$.  Therefore
\cref{thm:Vu} applies whenever
\[
 \frac{\rho p}{4}\ge
 K_V\frac d{\Log(d^2/Y)},
 \qquad\text{equivalently}\qquad
 p\ge\frac{4K_V}{\rho}\frac d{\Log(d^2/Y)}.
\]
Together with \eqref{eq:denseVertexRetired}, the entire procedure succeeds
whenever
\begin{equation}\label{eq:denseVertexGeneric}
 p\ge C_\rho\left(
 \frac d{\Log(d^2/Y)}+\frac{nY}{d^2}+\frac nd
 \right).
\end{equation}

Put $N=\Log n$.  If $Y\le\sqrt{nN}$, take $d=\sqrt{nN}$.  Then
$\Log(d^2/Y)\ge cN$, and every term in
\eqref{eq:denseVertexGeneric} is
$O(\sqrt{n/N})$.

Now suppose $Y>\sqrt{nN}$, put $f=\Log(n^2/Y)$, and take
$d=(nYf)^{1/3}$.  We have
\[
 \Log(d^2/Y)
 =\Log\bigl((n^2/Y)^{1/3}f^{2/3}\bigr)\ge cf.
\]
Thus the first two terms in \eqref{eq:denseVertexGeneric} are at most
\[
 C\frac{n^{1/3}Y^{1/3}}{f^{2/3}}.
\]
By \cref{lem:logMonotonicity},
\[
 Yf\ge \sqrt{nN}\,\Log\left(\frac{n^2}{\sqrt{nN}}\right)
 \ge c n^{1/2}N^{3/2},
\]
and hence
\[
 \frac nd=\frac{n^{2/3}}{(Yf)^{1/3}}
 \le C\sqrt{\frac nN}.
\]
Choose $K=K_{\mathrm{vert}}(\rho,A)$ to dominate the constants in the two
optimized bounds above, and choose $p_0=p_{\mathrm{vert}}(\rho,A)$ at least
$p_{\mathrm{res}}(\rho,A)$ and large enough for the bounded values of $n$.
Then \eqref{eq:denseVertexGeneric} follows from the displayed hypothesis,
which proves the lemma.
\end{proof}

\begin{lemma}[Dense layered estimate]\label{lem:denseLayered}
For every $\rho,A>0$ there are constants $K_{\rho,A}$ and $p_0=p_0(\rho,A)$ with the following property.  Suppose $p\ge p_0$, $V(F)=A_1\sqcup\cdots\sqcup A_k$, and $F$ has
local triangle bound $y$, $|F|\le p^A$, $1\le d\le p^A$, and
\[
 |N(v)\cap A_j|\le d x^{i-j}
 \qquad(v\in A_i,\ 1\le i\le j\le k)
\]
for $x\ge\sqrt2$.  If the lists are $\rho$-dense in a common palette of
size $p$ and
\[
 p\ge K_{\rho,A}
 \frac{d\left(1+\dfrac{\LogLog(d^2/y_+)}{\log x}\right)}
      {\Log(d^2/y_+)},
\]
then $F$ is list-colorable.
\end{lemma}

\begin{proof}
The proof follows the layered architecture of Harris's Lemma~2.2 \cite{Harris} in the common-palette list setting.  Put $f=\Log(d^2/y_+)$.  First suppose $x\ge2f$.  Color
$A_k,A_{k-1},\dots,A_1$ in this order.  Before coloring $A_j$, a vertex
of $A_j$ has at most
\[
 d(x^{-1}+x^{-2}+\cdots)\le\frac d f
\]
already colored neighbors.  Let $L_j(v)$ be its list after deleting the
colors used by those neighbors.  Then
\begin{equation}\label{eq:layeredResidualList}
 |L_j(v)|\ge \rho p-\frac d f.
\end{equation}
The graph $F[A_j]$ has maximum degree at most $d$ and local triangle
bound $y$.  Hence the hypothesis of \cref{thm:Vu} is satisfied whenever
\[
 \rho p-\frac d f\ge K_V\frac d f,
 \qquad\text{that is, whenever}\qquad
 p\ge\frac{K_V+1}{\rho}\frac d f.
\]
This proves the first case.

Suppose now that $x<2f$, and put
\[
 q=\left\lceil\frac{\log(2f)}{\log x}\right\rceil.
\]
Since $d\le p^A$ and $y_+\ge1$, we have
$f\le C_A\Log p$, and therefore
$q=O_A(\LogLog p)\le\Log p$ for all sufficiently large $p$.  Group the
layers by their indices modulo $q$ and reserve $q$ disjoint subpalettes
using \cref{lem:reservation}.  Within each residue class, color the layers
in decreasing order of their indices; distinct residue classes use their
reserved disjoint subpalettes.  Fix one residue class and one of its
layers.  The effective ratio inside that class is $x^q\ge2f$, so at most
$d/f$ colors have been used by previously colored neighbors.  The
reservation lemma gives at least $\rho p/(2q)$ available colors before
these deletions.  Thus the remaining restricted list has size at least
\begin{equation}\label{eq:layeredReservedResidualList}
 \frac{\rho p}{2q}-\frac d f.
\end{equation}
It meets Vu's requirement once
\[
 \frac{\rho p}{2q}-\frac d f\ge K_V\frac d f,
 \qquad\text{equivalently}\qquad
 p\ge\frac{2q(K_V+1)}{\rho}\frac d f.
\]
Since
\[
 q\le 1+\frac{\log(2f)}{\log x}
 \le C_A\left(1+\frac{\LogLog(d^2/y_+)}{\log x}\right),
\]
the displayed hypothesis of the lemma implies this inequality after
increasing $K_{\rho,A}$.  Take $K_{\rho,A}$ to dominate the constants in the two cases, and take
$p_0(\rho,A)\ge p_{\mathrm{res}}(\rho,A)$ large enough that
$q\le\Log p$ whenever the second case occurs.  Together with
\eqref{eq:layeredReservedResidualList}, this completes the coloring of the
layers.
\end{proof}

\subsection{Dense triangle-sensitive estimates}

\begin{lemma}[Dense vertex-triangle estimate]\label{lem:denseVertexTriangle}
For every $\rho,A>0$ there are constants $K=K(\rho,A)$ and $p_0=p_0(\rho,A)$ such that the following holds whenever $p\ge p_0$.  Let $F$ have $n\le p^A$ vertices,
$t$ triangles, and local triangle bound $y$.  If the lists have density
at least $\rho$ in a common palette of size $p$ and
\[
 K\left(
 \sqrt{\frac n{\Log n}}
 +t^{1/3}\frac{\LogLog(t^2/y_+^3)}
                    {\Log^{2/3}(t^2/y_+^3)}
 \right)\le p,
\]
then $F$ is list-colorable.
\end{lemma}

\begin{proof}
We prove a common-palette list analogue of Harris's Theorem~2.3
\cite{Harris}, using separate palettes for its three phases.  Put
\[
 R_{\mathrm{VT}}=\Phi_V(n)+\Phi_T(t,y).
\]
Let $K_{\mathrm{zero}},p_{\mathrm{zero}}$ be the constants obtained from
\cref{lem:denseVertex} with parameters $(\rho,A)$, enlarged by an absolute
factor; this covers $t=0$ because
$\Phi_U(n,0)\le C\Phi_V(n)$.  Assume $t>0$, and replace
$y_+$ by $Y=\min\{y_+,t\}$.  This remains a local triangle bound, while
\cref{lem:controlMonotonicity} gives
$\Phi_T(t,Y)\le C\Phi_T(t,y)$.  Put
\[
 f=\Log(t^2/Y^3),
 \qquad
 N=\Log n,
 \qquad
 T=t^{1/3}\frac{\Log f}{f^{2/3}},
\]
\[
 d=(ft)^{1/3}+\sqrt{\frac nN},
 \qquad
 h_0=\lceil\log_2d\rceil,
 \qquad
 \ell=\lceil\log_2Y\rceil.
\]
Let $K_{\mathrm{lay}},p_{\mathrm{lay}}$ be the constants in
\cref{lem:denseLayered} with parameters $(\rho/4,2A+1)$, and let
$K_{\mathrm{vert}},p_{\mathrm{vert}}$ be the constants in
\cref{lem:denseVertex} with parameters $(\rho/4,2A)$.

Reserve three top-level palette blocks $P_D,P_H,P_L$.  By
\cref{lem:reservation}, the blocks are pairwise disjoint, their restricted
lists have density at least $\rho/4$, each block has order
$\Theta_\rho(p)$, and $n\le |P_*|^{2A}$.

After each deletion, recompute the layers in the current graph.  Thus the
current vertices are partitioned into $A_{-1},A_0,\dots,A_\ell$, where
$A_{-1}$ consists of vertices in no current triangle and $A_i$ consists of
vertices whose current local triangle count lies in $[2^i,2^{i+1})$.
Suppose there are $h_0\le i\le j\le\ell$ and $v\in A_i$ with
\[
 S=N(v)\cap A_j,
 \qquad
 |S|>2^{(i-j)/2}d.
\]
The graph on $S$ has at most $2^{i+1}$ edges.  A color used from $P_D$ is
retired for the remainder of the deletion phase.  As long as fewer than
$(\rho/32)|P_D|$ colors have been retired, every current deletion list has
density at least $\rho/8$ in $P_D$.  Applying
\cref{lem:weightedIndependent} gives an independent set $I\subseteq S$
sharing one unretired color and satisfying
\[
 |I|\ge c_\rho\frac{|S|^2}{2^{i+2}+|S|}.
\]
The function $x\mapsto x^2/(2^{i+2}+x)$ is increasing on $x>0$, since
its derivative is
\[
 \frac{x(2^{i+3}+x)}{(2^{i+2}+x)^2}>0.
\]
Put $x_0=2^{(i-j)/2}d$.  Since $|S|>x_0$ and $x_0\le d$, while
$2^i\ge d$, we obtain
\begin{equation}\label{eq:denseTriangleIndependent}
 |I|\ge c_\rho
 \frac{2^{i-j}d^2}{2^{i+2}+d}
 \ge c_\rho'2^{-j}d^2.
\end{equation}
Every vertex of $I$ lies in at least $2^j$ current triangles, and no
current triangle contains two vertices of $I$.  Deleting $I$ therefore
removes at least $c_\rho'd^2$ current triangles.  Since every such triangle
contains a vertex deleted at that step, the triangles removed at different
steps are disjoint.  The number of retired colors in this phase is at most
\begin{equation}\label{eq:denseTriangleRetired}
 C_\rho\frac t{d^2}.
\end{equation}
As $|P_D|\ge p/6$, the deletion phase is valid whenever
\begin{equation}\label{eq:VTDeletionRequirement}
 p\ge\frac{192C_\rho}{\rho}\frac t{d^2}.
\end{equation}
The choice of $d$ gives
\begin{equation}\label{eq:VTDeletionControl}
 \frac t{d^2}\le\frac{t^{1/3}}{f^{2/3}}\le T\le R_{\mathrm{VT}}.
\end{equation}

When no violating pair remains, let $F^\circ$ be the current graph and let
$F_H$ be the subgraph induced by
$\bigcup_{i=h_0}^{\ell}A_i$, with this union interpreted as empty when
$h_0>\ell$.  The chosen $d$ satisfies
$d\le Cn\Log^{1/3}n$, because $t\le n^3$ and $f\le C\Log n$.
Since $n\le |P_H|^{2A}$, we have $d\le |P_H|^{2A+1}$ for all
sufficiently large $p$.  Define
\[
 R_H=\frac{d\left(1+\dfrac{\LogLog(d^2/Y)}{\log\sqrt2}\right)}
           {\Log(d^2/Y)}.
\]
Put $s_H=\Log(d^2/Y)$.  Since
\[
 s_H\ge c\Log\bigl((t^2/Y^3)^{1/3}f^{2/3}\bigr)\ge cf,
\]
the function $g(s)=(1+\Log s)/s$ is nonincreasing for $s\ge1$: on
$[1,e]$ it equals $2/s$, and for $s\ge e$ its derivative is
$-(\log s)/s^2$.  Since $g$ is nonincreasing,
\[
 g(s_H)\le g(\max\{1,cf\})
 \le Cg(f)
 \le C\frac{\Log f}{f},
\]
where the middle inequality absorbs the bounded range of $f$ and the
fixed dilation by $c$.
Consequently
\begin{align}
 R_H
 &\le C d\frac{\Log f}{f}\notag\\
 &\le C\left(
       t^{1/3}\frac{\Log f}{f^{2/3}}
       +\sqrt{\frac nN}\frac{\Log f}{f}
       \right)\notag\\
 &\le C\left(T+\sqrt{\frac nN}\right)
 \le B_{\mathrm{VT,L}}R_{\mathrm{VT}}.
 \label{eq:VTLayeredControl}
\end{align}

Let $F_L$ be the subgraph induced by
\[
 A_{-1}\cup\bigcup_{i=0}^{\min\{h_0-1,\ell\}}A_i.
\]
Then
\[
 V(F^\circ)=V(F_H)\sqcup V(F_L).
\]
The graph $F_L$ has at most $n$ vertices and local triangle bound less
than $2^{h_0}<2d$.  By \cref{lem:controlMonotonicity},
\begin{align*}
 R_L
 &: =\Phi_V(|F_L|)+\Phi_U(|F_L|,2d)\\
 &\le C\left(
       \sqrt{\frac nN}
       +\frac{(nd)^{1/3}}{\Log^{2/3}(n^2/d)}
       \right).
\end{align*}
Since $d\le CnN^{1/3}$, we have $\Log(n^2/d)\ge cN$, and hence
\begin{equation}\label{eq:lowTriangleAlgebra}
 \frac{(nd)^{1/3}}{\Log^{2/3}(n^2/d)}
 \le C\left(
 \frac{\sqrt n}{N^{5/6}}
 +\frac{n^{1/3}t^{1/9}f^{1/9}}{N^{2/3}}
 \right).
\end{equation}
The first term is at most $\sqrt{n/N}$.  If
$t\le n^{3/2}N^{1/2}$, the second term is also at most
$C\sqrt{n/N}$, using $f\le CN$.  If
$t>n^{3/2}N^{1/2}$, then
\[
 \frac{n^{1/3}t^{1/9}f^{1/9}}{N^{2/3}}
 =\frac{t^{1/3}}{f^{2/3}}
   \frac{f^{7/9}n^{1/3}}{t^{2/9}N^{2/3}}
 \le C\frac{t^{1/3}}{f^{2/3}}
 \le CT.
\]
Thus
\begin{equation}\label{eq:VTVertexControl}
 R_L\le C\left(T+\sqrt{\frac nN}\right)
 \le B_{\mathrm{VT,V}}R_{\mathrm{VT}}.
\end{equation}

Choose
\[
 K_{\mathrm{VT}}
 \ge\max\left\{
   K_{\mathrm{zero}},
   \frac{192C_\rho}{\rho},
   6B_{\mathrm{VT,L}}K_{\mathrm{lay}},
   6B_{\mathrm{VT,V}}K_{\mathrm{vert}}
 \right\},
\]
and choose
\[
 p_{\mathrm{VT}}
 \ge\max\left\{
 p_{\mathrm{zero}},
 p_{\mathrm{res}}(\rho,A),
 6p_{\mathrm{lay}},
 6p_{\mathrm{vert}}
 \right\}
\]
large enough also for $d\le |P_H|^{2A+1}$ and the bounded parameter
ranges above.  If $p\ge p_{\mathrm{VT}}$ and
$p\ge K_{\mathrm{VT}}R_{\mathrm{VT}}$, then
\eqref{eq:VTDeletionRequirement} holds, while
\cref{lem:parameterInheritance} with $q_0=3$ and
\eqref{eq:VTLayeredControl}--\eqref{eq:VTVertexControl} verifies the
hypotheses of \cref{lem:denseLayered,lem:denseVertex} on $P_H$ and $P_L$.
These two colorings, together with the deletion coloring in $P_D$, color
$F$ because the three palette blocks are disjoint.  Taking
$K=K_{\mathrm{VT}}$ and $p_0=p_{\mathrm{VT}}$ proves the lemma.
\end{proof}

\begin{lemma}[Dense edge-local estimate]\label{lem:denseEdgeLocal}
For every $\rho,A>0$ there are constants $K=K(\rho,A)$ and $p_0=p_0(\rho,A)$ such that the following holds whenever $p\ge p_0$.  Let $F$ have at most $p^A$
nonisolated vertices, $m$ edges, and local triangle bound $y$.  If the
lists have density at least $\rho$ in a common palette of size $p$ and
\[
 K\left(
 \frac{m^{1/3}}{\Log^{2/3}m}
 +\frac{m^{1/4}y_+^{1/4}}{\Log^{3/4}(m/y_+)}
 \right)\le p,
\]
then $F$ is list-colorable.
\end{lemma}

\begin{proof}
We prove a common-palette list analogue of Harris's Theorem~3.1 \cite{Harris}.  Put
\[
 R_{\mathrm{EL}}=\Phi_E(m)+\Phi_W(m,y).
\]
Let $K_{\mathrm{vert}},p_{\mathrm{vert}}$ be the constants in
\cref{lem:denseVertex} with parameters $(\rho/4,2A)$.
First discard the isolated vertices; after the nonisolated part has been
colored, each isolated vertex may be assigned any color from its nonempty
list.  We may therefore assume that $|F|\le p^A$.  The case $m=0$ is then
trivial.  Replace $y_+$ by
$Y=\min\{y_+,m\}$.  This remains a local triangle bound, and
$\Phi_W(m,Y)\le C\Phi_W(m,y)$ by
\cref{lem:controlMonotonicity}.  Reserve two palette blocks $P_{\mathrm{hi}},P_{\mathrm{lo}}$.  For a
parameter $d\ge1$, let $V_{\mathrm{hi}}$ be the vertices of degree greater
than $d$, and put $V_{\mathrm{lo}}=V(F)\setminus V_{\mathrm{hi}}$.  Then
$|V_{\mathrm{hi}}|\le2m/d$.  Write $n_{\mathrm{hi}}=|V_{\mathrm{hi}}|$
and let $y_{\mathrm{hi}}$ be the local triangle bound of
$F[V_{\mathrm{hi}}]$.  Thus
\[
 n_{\mathrm{hi}}\le2m/d,
 \qquad
 y_{\mathrm{hi}}\le Y.
\]
By \cref{lem:controlMonotonicity},
\[
 \Phi_V(n_{\mathrm{hi}})+\Phi_U(n_{\mathrm{hi}},y_{\mathrm{hi}})
 \le C\bigl(\Phi_V(m/d)+\Phi_U(m/d,Y)\bigr).
\]
Put
\[
 R_{\mathrm{hi}}
 =\Phi_V(n_{\mathrm{hi}})+\Phi_U(n_{\mathrm{hi}},y_{\mathrm{hi}}).
\]
The high-degree subgraph can be colored from $P_{\mathrm{hi}}$ provided
\[
 p\ge4K_{\mathrm{vert}}R_{\mathrm{hi}},
 \qquad
 p\ge4p_{\mathrm{vert}}.
\]
The graph $F[V_{\mathrm{lo}}]$ has maximum degree at most $d$, and its
restricted lists in $P_{\mathrm{lo}}$ have size at least $\rho p/4$.
Thus \cref{thm:Vu} applies to $F[V_{\mathrm{lo}}]$ whenever
\[
 p\ge\frac{4K_V}{\rho}\frac d{\Log(d^2/Y)}.
\]
The analytic quantities to be controlled are therefore
\begin{equation}\label{eq:denseEdgeLocalGeneric}
 R_{\mathrm{hi}}
 \quad\text{and}\quad
 \frac d{\Log(d^2/Y)}.
\end{equation}

Put $L=\Log m$.  If $Y\le(mL)^{1/3}$, take
$d=(mL)^{1/3}$.  Then
\[
 \Log(m/d),\quad
 \Log((m/d)^2/Y),\quad
 \Log(d^2/Y)
 \ge cL.
\]
Consequently
\[
 R_{\mathrm{hi}}
 \le C\frac{m^{1/3}}{L^{2/3}},
 \qquad
 \frac d{\Log(d^2/Y)}
 \le C\frac{m^{1/3}}{L^{2/3}}.
\]

Suppose $Y>(mL)^{1/3}$.  Put $f=\Log(m/Y)$ and
$d=(mYf)^{1/4}$.  By \cref{lem:logMonotonicity},
\[
 Yf\ge c m^{1/3}L^{4/3},
 \qquad
 Yf\le Cm.
\]
Thus
\[
 c m^{1/3}L^{1/3}\le d\le C\sqrt m,
\]
so $\Log(m/d)\ge cL$.  Directly from the definition of $d$,
\[
 \Log(d^2/Y)\ge cf,
 \qquad
 \Log((m/d)^2/Y)\ge cf.
\]
The second term in $R_{\mathrm{hi}}$ and the low-degree quantity in
\eqref{eq:denseEdgeLocalGeneric} are consequently both at most
\[
 C\frac{m^{1/4}Y^{1/4}}{f^{3/4}},
\]
while the first term in $R_{\mathrm{hi}}$ is at most
$Cm^{1/3}/L^{2/3}$.  Since
$\Phi_W(m,Y)\le C\Phi_W(m,y)$, both ranges give
\begin{equation}\label{eq:ELVertexControl}
 R_{\mathrm{hi}}\le B_{\mathrm{EL,V}}R_{\mathrm{EL}},
 \qquad
 \frac d{\Log(d^2/Y)}\le B'_{\mathrm{EL}}R_{\mathrm{EL}}.
\end{equation}
Choose
\[
 K_{\mathrm{EL}}
 \ge4\max\left\{
 B_{\mathrm{EL,V}}K_{\mathrm{vert}},
 \frac{B'_{\mathrm{EL}}K_V}{\rho}
 \right\},
\]
and choose
$ p_{\mathrm{EL}}\ge\max\{p_{\mathrm{res}}(\rho,A),4p_{\mathrm{vert}}\}$
large enough for the bounded parameter ranges.  Then
$p\ge K_{\mathrm{EL}}R_{\mathrm{EL}}$ and
$p\ge p_{\mathrm{EL}}$ meet the two coloring conditions above, so the
disjoint palettes $P_{\mathrm{hi}}$ and $P_{\mathrm{lo}}$ color $F$.
Taking $K=K_{\mathrm{EL}}$ and $p_0=p_{\mathrm{EL}}$ proves the lemma.
\end{proof}

\begin{lemma}[Dense edge-triangle estimate]\label{lem:denseEdgeTriangle}
For every $\rho,A>0$ there are constants $K=K(\rho,A)$ and $p_0=p_0(\rho,A)$ such that the following holds whenever $p\ge p_0$.  Let $F$ have at most $p^A$
nonisolated vertices, $m$ edges, $t$ triangles, and local triangle bound
$y$.  If the lists have density at least $\rho$ in a common palette of
size $p$ and
\[
 K\left(
 \frac{m^{1/3}}{\Log^{2/3}m}
 +t^{1/3}\frac{\LogLog(t^2/y_+^3)}
                    {\Log^{2/3}(t^2/y_+^3)}
 \right)\le p,
\]
then $F$ is list-colorable.
\end{lemma}

\begin{proof}
We prove a common-palette list analogue of Harris's Theorem~3.2 \cite{Harris} from the preceding dense estimates.  Put
\[
 R_{\mathrm{ET}}=\Phi_E(m)+\Phi_T(t,y).
\]
Let $K_{\mathrm{VT}},p_{\mathrm{VT}}$ and
$K_{\mathrm{EL}},p_{\mathrm{EL}}$ be the constants supplied by
\cref{lem:denseVertexTriangle,lem:denseEdgeLocal}, respectively, with
parameters $(\rho/4,2A)$.  Let
$K_{\mathrm{zero}},p_{\mathrm{zero}}$ be the constants obtained from
\cref{lem:denseEdgeLocal} with parameters $(\rho,A)$, enlarged by an
absolute factor; this covers $t=0$ because
$\Phi_W(m,0)\le C\Phi_E(m)$.
Discard isolated vertices at the outset and color them after the
nonisolated part from their nonempty lists; hence $|F|\le p^A$ throughout
the palette argument.  Assume $t>0$, and replace
$y_+$ by
$Y=\min\{y_+,t\}$.  This remains a local triangle bound, and
$\Phi_T(t,Y)\le C\Phi_T(t,y)$ by
\cref{lem:controlMonotonicity}.  Put
\[
 L=\Log m,
 \qquad
 f=\Log(t^2/Y^3),
 \qquad
 z_0=t^{2/3}\frac{L^{1/3}}{m^{1/3}},
 \qquad
 z=\max\{1,z_0\}.
\]
Let $U$ be the vertices lying in at least $z$ triangles and put
$W=V(F)\setminus U$.  Let $n_U,t_U,y_U$ be the order, triangle count,
and local triangle bound of $F[U]$, and let $m_W,y_W$ be the edge count
and local triangle bound of $F[W]$.  Since the sum of the local triangle
counts is $3t$,
\[
 n_U\le3t/z,
 \qquad t_U\le t,
 \qquad y_U\le Y,
 \qquad m_W\le m,
 \qquad y_W\le z.
\]
Consequently \cref{lem:controlMonotonicity} gives
\begin{align*}
 \Phi_V(n_U)&\le C\Phi_V(t/z),
 &\Phi_T(t_U,y_U)&\le C\Phi_T(t,Y),\\
 \Phi_E(m_W)&\le C\Phi_E(m),
 &\Phi_W(m_W,y_W)&\le C\Phi_W(m,z).
\end{align*}
Reserve two top-level palette blocks $P_U,P_W$.  Put
\[
 R_U=\Phi_V(n_U)+\Phi_T(t_U,y_U),
 \qquad
 R_W=\Phi_E(m_W)+\Phi_W(m_W,y_W).
\]
The two child estimates apply from the reserved palettes provided
\begin{equation}\label{eq:denseEdgeTriangleChildConditions}
 p\ge4K_{\mathrm{VT}}R_U,
 \qquad
 p\ge4K_{\mathrm{EL}}R_W,
 \qquad
 p\ge4\max\{p_{\mathrm{VT}},p_{\mathrm{EL}}\}.
\end{equation}
The remaining work is to compare $R_U$ and $R_W$ with
$R_{\mathrm{ET}}$.

Every graph with $m$ edges has $t\le C m^{3/2}$.  Indeed, if
$\lambda_1,\dots,\lambda_n$ are the adjacency eigenvalues and
$\mathsf A$ is the adjacency matrix of $F$, then
\[
 6t=\operatorname{tr}(\mathsf A^3)
 \le\sum_i|\lambda_i|^3
 \le\left(\sum_i\lambda_i^2\right)^{3/2}
 =(2m)^{3/2}.
\]
If $z=1$, then $t^2L\le m$.  Hence
\[
 \Phi_V(t)
 =\sqrt{\frac{t}{\Log t}}
 \le\sqrt t
 \le\frac{m^{1/4}}{L^{1/4}}
 \le C\frac{m^{1/3}}{L^{2/3}},
\]
and
\[
 \Phi_W(m,1)=\frac{m^{1/4}}{L^{3/4}}
 \le C\frac{m^{1/3}}{L^{2/3}}.
\]
Together with the preliminary monotonicity comparisons, this gives the
required bounds for $R_U$ and $R_W$ in the range $z=1$.

Suppose $z=z_0\ge1$.  From $t^2L\ge m$ and $t\le Cm^{3/2}$,
\[
 \Log(t/z)\ge cL,
 \qquad
 \Log(m/z)\ge cL.
\]
Therefore
\begin{align*}
 \Phi_V(t/z)
 &\le C\frac{m^{1/6}t^{1/6}}{L^{2/3}},\\
 \Phi_W(m,z)
 &\le C\frac{m^{1/6}t^{1/6}}{L^{2/3}}.
\end{align*}
By the arithmetic--geometric mean inequality,
\[
 \frac{m^{1/6}t^{1/6}}{L^{2/3}}
 \le C\left(
 \frac{m^{1/3}}{L^{2/3}}
 +\frac{t^{1/3}}{L^{2/3}}
 \right).
\]
Finally $f\le CL$, so
\[
 \frac{t^{1/3}}{L^{2/3}}
 \le C\frac{t^{1/3}}{f^{2/3}}
 \le C t^{1/3}\frac{\Log f}{f^{2/3}}.
\]
Thus in both ranges
\begin{align}
 R_U&\le B_{\mathrm{ET,VT}}R_{\mathrm{ET}},\label{eq:ETVTControl}\\
 R_W&\le B_{\mathrm{ET,EL}}R_{\mathrm{ET}}.\label{eq:ETELControl}
\end{align}
Choose
\[
 K_{\mathrm{ET}}
 \ge\max\left\{
 K_{\mathrm{zero}},
 4B_{\mathrm{ET,VT}}K_{\mathrm{VT}},
 4B_{\mathrm{ET,EL}}K_{\mathrm{EL}}
 \right\},
\]
and choose
\[
 p_{\mathrm{ET}}
 \ge\max\left\{
 p_{\mathrm{zero}},
 p_{\mathrm{res}}(\rho,A),
 4p_{\mathrm{VT}},
 4p_{\mathrm{EL}}
 \right\}
\]
large enough for the bounded parameter ranges.  Then
$p\ge K_{\mathrm{ET}}R_{\mathrm{ET}}$ and $p\ge p_{\mathrm{ET}}$
imply \eqref{eq:denseEdgeTriangleChildConditions}, and the two disjoint
palette blocks color $F$.  Taking $K=K_{\mathrm{ET}}$ and
$p_0=p_{\mathrm{ET}}$ proves the lemma.
\end{proof}

The finite induction used in the preceding proofs is summarized below.
The control constants in the last column are supplied by the indicated
estimates.

\begin{table}[htbp]
\centering
\small
\setlength{\tabcolsep}{5pt}
\renewcommand{\arraystretch}{1.2}
\caption{Recursive calls in the dense common-palette estimates.}
\label{tab:recursiveCalls}
\begin{tabular}{@{}c c c p{0.42\textwidth}@{}}
\hline
call & child palette & child parameters & control comparison \\
\hline
$\mathrm{VT}\to\mathrm L$
 & $\ge p/6$ & $(\rho/4,2A+1)$
 & $R_H\le B_{\mathrm{VT,L}}R_{\mathrm{VT}}$
   by \eqref{eq:VTLayeredControl} \\
$\mathrm{VT}\to\mathrm V$
 & $\ge p/6$ & $(\rho/4,2A)$
 & $R_L\le B_{\mathrm{VT,V}}R_{\mathrm{VT}}$
   by \eqref{eq:VTVertexControl} \\
$\mathrm{EL}\to\mathrm V$
 & $\ge p/4$ & $(\rho/4,2A)$
 & $R_{\mathrm{hi}}\le B_{\mathrm{EL,V}}R_{\mathrm{EL}}$
   by \eqref{eq:ELVertexControl} \\
$\mathrm{ET}\to\mathrm{VT}$
 & $\ge p/4$ & $(\rho/4,2A)$
 & $R_U\le B_{\mathrm{ET,VT}}R_{\mathrm{ET}}$
   by \eqref{eq:ETVTControl} \\
$\mathrm{ET}\to\mathrm{EL}$
 & $\ge p/4$ & $(\rho/4,2A)$
 & $R_W\le B_{\mathrm{ET,EL}}R_{\mathrm{ET}}$
   by \eqref{eq:ETELControl} \\
\hline
\end{tabular}
\end{table}

\begin{proof}[Proof of \cref{thm:densePalette}]
Let $K,p_0$ be the constants supplied by
\cref{lem:denseEdgeTriangle} for the parameters $(\rho,A)$, enlarged by
an absolute factor if necessary.  Apply that lemma with local triangle
bound $y=t$.
If $t=0$, its second term is zero.  If $t\ge1$, then
$t^2/y_+^3=1/t\le1$, so our truncated logarithms make the second term
exactly $t^{1/3}$.  Enlarging the constant gives the stated theorem.
\end{proof}

\section{The exact near-clique regime}\label{sec:near}

We prove \cref{thm:nearClique}.

\begin{lemma}[Natural lists]\label{lem:naturalLists}
Let $Q$ be a clique of order $q=r-j$ in a graph $G$, where $1\le j<r$.  Give the vertices of $Q$ distinct colors, and let $X$ be a disjoint set of $j-1$ further colors.  For $v\notin Q$, put
\[
 L(v)=X\cup\bigl(Q\setminus N_Q(v)\bigr),
\]
where a vertex of $Q$ denotes its assigned color.  If $G-Q$ is $L$-colorable, then $G$ is $(r-1)$-colorable.

If $U\subseteq G-Q$ is not $L$-colorable, then $\chi(G[Q\cup U])\ge r$.
\end{lemma}

\begin{proof}
The first assertion is immediate: a color from $Q$ can be reused at $v$ exactly when $v$ is not adjacent to the corresponding vertex of $Q$.  For the second, any coloring of $G[Q\cup U]$ with at most $r-1$ colors uses $q$ distinct colors on $Q$ and at most $j-1$ additional colors.  Relabeling gives an $L$-coloring of $U$, a contradiction.
\end{proof}

\begin{proof}[Proof of \cref{thm:nearClique}]
Let $K_*$ be the constant in \cref{thm:densePalette} for $\rho=1/3$ and $A=4$.  Choose $\delta_0>0$ and then $c_0>0$ sufficiently small in terms of $K_*$.  Suppose, for a contradiction, that
\[
 \chi(G)\ge r,
 \qquad
 e(G)\le c_0r^3\log^2r,
 \qquad
 \omega(G)\ge(1-\delta_0)r,
 \qquad
 \tri(G)<\binom r3.
\]
If $G$ contains $K_r$, the last inequality is already impossible.  Thus
$q:=\omega(G)\le r-1$.  Let $Q$ be a maximum clique and put $j=r-q$.
Then
\[
 q=r-j\ge(1-\delta_0)r,
 \qquad
 1\le j\le\delta_0r.
\]
Use the natural lists from \cref{lem:naturalLists}.  For $v\notin Q$, write $a_v=|N_Q(v)|$.  The total number of triangles not contained in $Q$ is less than
\begin{equation}\label{eq:deficit}
 D:=\binom r3-\binom q3
 =\sum_{u=q}^{r-1}\binom u2
 \le j\binom r2
 \le\frac{\delta_0r^3}{2}.
\end{equation}
Every $v\notin Q$ creates $\binom{a_v}{2}$ triangles using two vertices of $Q$.

Let
\[
 Z=\{v\notin Q:a_v\ge q/2\}.
\]
Then
\[
 |Z|\binom{\lfloor q/2\rfloor}{2}<D,
\]
and hence, for sufficiently small $\delta_0$ and large $r$,
\begin{equation}\label{eq:Zbound}
 |Z|\le10\delta_0r.
\end{equation}
We claim that $G[Z]$ is $L$-colorable.  Otherwise take an inclusion-minimal non-$L$-colorable $U\subseteq Z$.  By \cref{lem:naturalLists}, $\chi(G[Q\cup U])\ge r$, while
\[
 |Q\cup U|\le(1+10\delta_0)r\le2r.
\]
Taking an induced $r$-critical subgraph and applying \cref{thm:FTZlinear} gives at least $\binom r3$ triangles, a contradiction.

Fix an $L$-coloring of $G[Z]$ and let $H=G-(Q\cup Z)$.  From each list on $H$, delete the colors used by its neighbors in $Z$; call the new list $L'(v)$.  Since $a_v<q/2$ on $H$, \eqref{eq:Zbound} gives
\[
 |L'(v)|
 \ge r-1-q/2-|Z|
 \ge(1/2-10\delta_0)r-1
 \ge\frac{r-1}{3}
\]
for sufficiently small $\delta_0$.  Thus the lists have density at least $1/3$ in the common palette of $p=r-1$ colors.

Discard isolated vertices temporarily.  The remaining graph has at most
\[
 2e(H)\le2c_0r^3\log^2r\le p^4
\]
vertices.  Moreover,
\[
 e(H)\le c_0r^3\log^2r,
 \qquad
 \tri(H)<D\le\frac{\delta_0r^3}{2}.
\]
Consequently, by \cref{lem:controlMonotonicity} and direct evaluation at
$c_0r^3\log^2r$,
\[
 \frac{e(H)^{1/3}}{\Log^{2/3}e(H)}
 \le C\Phi_E(c_0r^3\log^2r)
 =O(c_0^{1/3}r)+O(1),
 \qquad
 \tri(H)^{1/3}=O(\delta_0^{1/3}r).
\]
Choose $\delta_0$ and $c_0$ so that \cref{thm:densePalette} applies.  It gives an $L'$-coloring of the nonisolated part of $H$.  Since every remaining list has size at least $(r-1)/3\ge1$, each isolated vertex can then be assigned an arbitrary color from its own list.  Together with the coloring of $Z$, this is an $L$-coloring of $G-Q$, so \cref{lem:naturalLists} gives an $(r-1)$-coloring of $G$.  This contradicts $\chi(G)\ge r$.
\end{proof}

\section{Triangle stability for critical graphs}\label{sec:stability}

We first prove a stability refinement of the linear-order theorem.  Its surplus may depend on the order constant.  The next subsection removes that dependence.

\subsection{Fixed-order stability}

\begin{proposition}\label{prop:fixedStability}
For every $C>0$ and $\beta>0$ there are $\gamma>0$ and $s_0$ such that every $s$-critical graph $G$ with $s\ge s_0$,
\[
 |G|\le Cs,
 \qquad
 \omega(G)\le(1-\beta)s
\]
satisfies
\[
 \tri(G)\ge\binom s3+\gamma s^3.
\]
\end{proposition}

\begin{proof}
It is enough to treat $0<\beta\le1/2$, since replacing a larger $\beta$
by $1/2$ only weakens the clique hypothesis.  Suppose the assertion is false.  Then, for every integer $k\ge1$, there
are $s_k\ge k$ and an $s_k$-critical graph $G_k$ satisfying the hypotheses
but
\[
 \tri(G_k)<\binom{s_k}{3}+\frac{s_k^3}{k}.
\]
Passing to this sequence and relabeling $s_k$ and $G_k$ as $s$ and $G$,
respectively, gives $s\to\infty$ and
\[
 \tri(G)<\binom s3+o(s^3).
\]
By \cref{thm:FTZlinear}, the reverse inequality
$\tri(G)\ge\binom s3$ holds for all sufficiently large members of the
sequence.  Hence
\begin{equation}\label{eq:nearExtremal}
 \tri(G)=\binom s3+o(s^3).
\end{equation}
We adapt the proof of Proposition~4.2 in \cite{FTZ}.

Set $G_0=G$ and take $G_0'=G$, which is already $s$-critical.  Having
defined an induced $(s-i)$-critical graph $G_i'\subseteq G_i$, let $I_i$
be a maximum independent set in $G_i'$, put $k_i=|I_i|$, and set
$G_{i+1}=G_i'-I_i$; then choose an induced $(s-i-1)$-critical graph
$G_{i+1}'\subseteq G_{i+1}$.  For every $i\ge0$, put
\[
 D_i=G_i\setminus G_i'.
\]
Thus $D_0=\varnothing$ and
\[
 G_i=D_i\sqcup I_i\sqcup G_{i+1}.
\]
In particular, for every index $a$ for which the construction is active,
\begin{equation}\label{eq:layerPartition}
 V(G_a)=\bigsqcup_{i=a}^{s-1}\bigl(D_i\sqcup I_i\bigr).
\end{equation}
Moreover, $\chi(G_{i+1})=s-i-1$, the sequence $k_i$ is nonincreasing,
and
\[
 \sum_i k_i\le|G|\le Cs.
\]
The edge and triangle decompositions in \cite[(4.1)--(4.2)]{FTZ} are
\begin{align}
 e(G)&\ge\sum_i\sum_{v\in I_i}d_{G_i'}(v)
 +\frac12\sum_i\sum_{v\in D_i}d_{G_i}(v),\label{eq:edgeDec}\\
 \tri(G)&\ge\sum_i\sum_{v\in I_i}e(G_i'[N(v)])
 +\frac13\sum_i\sum_{v\in D_i}e(G_i[N(v)]).\label{eq:triDec}
\end{align}
For $v\in I_i$, define
\[
 \ex(v)=e(G_i'[N(v)])+\frac12d_{G_i'}(v)
 -\frac{(s-i-1)^2}{2k_i},
\]
and, for $v\in D_i$, define
\[
 \ex(v)=\frac13e(G_i[N(v)])+\frac14d_{G_i}(v).
\]
Tur\'an's theorem gives
\begin{align}
 \ex(v)&\ge
 \frac{d_{G_i'}(v)^2-(s-i-1)^2}{2k_i}
 &&(v\in I_i),\label{eq:exI}\\
 \ex(v)&\ge\frac{d_{G_i}(v)^2}{6k_{i-1}}
 &&(i\ge1,\ v\in D_i).\label{eq:exD}
\end{align}
In particular all excesses are nonnegative.  Combining
\eqref{eq:edgeDec} and \eqref{eq:triDec}, and using
\[
 \frac12\sum_{j=0}^{s-1}j^2
 =\binom s3+\frac12\binom s2,
\]
we obtain
\begin{equation}\label{eq:totalExcess}
 E:=\sum_{v\in G}\ex(v)
 \le
 \tri(G)+\frac12e(G)
 -\binom s3-\frac12\binom s2
 =o(s^3),
\end{equation}
where $e(G)=O_C(s^2)$ and \eqref{eq:nearExtremal} were used.

Put
\[
 \theta=\frac\beta{64}
\]
and choose an integer $K>C/\theta$.  Let $\ell\ge1$ be the least index
with $k_{\ell-1}\le K$.  Since the preceding independent sets have
size greater than $K$,
\begin{equation}\label{eq:ell}
 \ell\le1+\theta s.
\end{equation}
Define the explicit error parameters
\[
 \epsilon_s=\frac{E}{s^3},
 \qquad
 \tau_s=\sqrt{\epsilon_s}+\frac1s,
 \qquad
 \xi_s=20\left(
  \frac{\epsilon_s}{\tau_s}+\sqrt{6K\tau_s}+\frac1s
 \right).
\]
Then $\epsilon_s,\tau_s,\xi_s\to0$.  Let $H$ be the subgraph of
$G_\ell$ induced by the vertices whose excess is at most $\tau_ss^2$.
At most $(\epsilon_s/\tau_s)s$ vertices are removed, and hence
\begin{equation}\label{eq:Hchi}
 \chi(H)\ge s-\ell-\frac{\epsilon_s}{\tau_s}s.
\end{equation}
For every $i\ge\ell$, we have $k_i,k_{i-1}\le K$.  Thus
\eqref{eq:exI}--\eqref{eq:exD} give the concrete bounds
\begin{align}
 d_{G_i'}(v)
 &\le s-i-1+\sqrt{2K\tau_s}\,s
 &&(v\in I_i\cap V(H)),\label{eq:degI}\\
 d_{G_i}(v)
 &\le\sqrt{6K\tau_s}\,s
 &&(v\in D_i\cap V(H)).\label{eq:degD}
\end{align}

Intersecting \eqref{eq:layerPartition} with $V(H)$ partitions $V(H)$
into the sets $I_i\cap V(H)$ and $D_i\cap V(H)$ for $i\ge\ell$.
Order these sets by decreasing layer index, listing $I_i\cap V(H)$ before
$D_i\cap V(H)$ inside layer $i$.  If $v\in I_i$, every earlier neighbor
lies in $I_i\cup G_{i+1}\subseteq G_i'$, and hence
$N^+(v)\subseteq G_i'[N(v)]$.  If $v\in D_i$, every earlier neighbor lies
in $D_i\cup I_i\cup G_{i+1}\subseteq G_i$, and hence
$N^+(v)\subseteq G_i[N(v)]$.  Set
\begin{equation}\label{eq:explicitDegeneracyParameter}
 D=\left\lceil s-\ell-1+\xi_ss\right\rceil.
\end{equation}
Since $\ell\le1+\theta s$ and $\xi_s\to0$, for all sufficiently
large $s$ we have
\begin{equation}\label{eq:Dbounds}
 (1-2\theta)s\le D\le2s.
\end{equation}
Moreover $\xi_s\ge20\sqrt{6K\tau_s}$, so
\eqref{eq:degI}--\eqref{eq:degD} and the definition of $D$ imply
$|N^+(v)|\le D$ for every $v\in H$.

Let $b$ be the first index with $k_b=1$.  Then $G_b'=K_{s-b}$, so the
clique hypothesis implies
\begin{equation}\label{eq:b}
 b\ge\beta s.
\end{equation}
Together with \eqref{eq:ell}, this gives $b>\ell$ for all sufficiently
large $s$.  Put
\[
 \delta=\frac\beta{16}.
\]
If $v\in I_i$ and $i<b$, then $k_i\ge2$, and the definition of excess
gives
\[
 e(H[N^+(v)])
 \le\tau_ss^2+\frac{(s-i-1)^2}{4}
 \le\left(\tau_s+\frac14\right)s^2.
\]
Here $\delta\le1/32$ and $\theta\le1/128$, so
\[
 \frac{(1-\delta)(1-2\theta)^2}{2}>\frac25.
\]
By \eqref{eq:Dbounds}, after increasing the lower threshold for $s$ and
requiring $\tau_s\le1/20$, the last display is at most
$(1-\delta)\binom D2$.

If $v\in I_i$ and $i\ge b$, then $G_i'$ is a clique and
\[
 |N^+(v)|\le s-b-1\le(1-\beta)s.
\]
Furthermore,
\[
 (1-\delta)(1-2\theta)^2
 \ge1-\delta-4\theta
 =1-\frac\beta8
 >(1-\beta)^2.
\]
Together with \eqref{eq:Dbounds}, this gives
$\binom{|N^+(v)|}{2}\le(1-\delta)\binom D2$ for all sufficiently large
$s$.  Finally, if $v\in D_i$, then the definition of excess gives
\[
 e(H[N^+(v)])\le3\tau_ss^2.
\]
The same lower bound $(1-\delta)\binom D2>(2/5+o(1))s^2$, together with
$\tau_s\to0$, proves the required inequality in this case as well.
Thus
\begin{equation}\label{eq:localSparse}
 e(H[N^+(v)])\le(1-\delta)\binom D2
 \qquad(v\in H).
\end{equation}

It remains to verify the degree drop in \cref{lem:degenerate}.  Put
\[
 h=\left\lfloor\frac{\delta D}{16K}\right\rfloor.
\]
Since
\[
 K(h+1)\le\frac{\delta D}{16}+K,
\]
we have $K(h+1)\le\delta D/8$ once $D\ge16K/\delta$.  There are no vertices of $H$ in $D_i$ for
$\ell\le i\le\ell+h$.  Indeed, for $v\in D_i$ we have
$G_i=G_{i-1}'-I_{i-1}$, while the induced $(s-i+1)$-critical graph
$G_{i-1}'$ has minimum degree at least $s-i$.  Hence
\begin{equation}\label{eq:discardedCoreDegree}
 d_{G_i}(v)
 \ge d_{G_{i-1}'}(v)-k_{i-1}
 \ge s-i-k_{i-1}.
\end{equation}
By \eqref{eq:ell}, \eqref{eq:Dbounds}, and
$h\le\delta D/(16K)$, we have $\ell+h+K\le s/2$ for all sufficiently
large $s$.  Thus the right-hand side of
\eqref{eq:discardedCoreDegree} is at least $s/2$, whereas
\eqref{eq:degD} is less than $s/2$ once
$\sqrt{6K\tau_s}<1/2$.  This is a contradiction.

Every vertex outside these final layers has either been discarded or
belongs to $I_i$ with $i\ge\ell+h+1$.  Put
\[
 \varepsilon_* = \frac{\delta}{64K}.
\]
For a discarded vertex, \eqref{eq:degD}, \eqref{eq:Dbounds}, and
$\tau_s\to0$ give explicitly
\[
 |N^+(v)|\le\sqrt{6K\tau_s}\,s
 \le(1-\varepsilon_*)D
\]
for all sufficiently large $s$.  For $v\in I_i$ with
$i\ge\ell+h+1$, \eqref{eq:degI} and the definition of $D$ give
\begin{align*}
 D-|N^+(v)|
 &\ge h+\xi_ss-\sqrt{2K\tau_s}\,s-O(1)\\
 &\ge \frac{\delta D}{16K}-O(1)\\
 &\ge\frac{\delta D}{32K}.
\end{align*}
The second line uses
$h\ge\delta D/(16K)-1$ and
$\xi_s\ge20\sqrt{6K\tau_s}\ge\sqrt{2K\tau_s}$; the last line follows
once $D$ is sufficiently large.  Hence all but the last
$K(h+1)\le\delta D/8<\delta D/4$ vertices of the order satisfy
\[
 |N^+(v)|\le(1-\varepsilon_*)D.
\]
All hypotheses of \cref{lem:degenerate} now hold with the explicit
integer parameter $D$.  It follows that
$\chi(H)\le(1-\zeta)D$ for some fixed $\zeta>0$.  Since
$\xi_ss\ge20$, the definition of $D$ gives
$D\ge s-\ell\ge\chi(H)$.  On the other hand,
\eqref{eq:Hchi} and \eqref{eq:explicitDegeneracyParameter} give
\[
 0\le D-\chi(H)
 \le\left(\xi_s+\frac{\epsilon_s}{\tau_s}\right)s+O(1)=o(s),
\]
while $D\ge(1-2\theta)s$.  Thus $\chi(H)/D\to1$, a contradiction.
This proves the proposition.
\end{proof}

\subsection{Removing the dependence on the order constant}

We isolate the one-step calculation used in the proof of Theorem~1.5 of \cite{FTZ}.

\begin{lemma}[Modified-weight deletion]\label{lem:modifiedWeight}
There is an absolute $\varepsilon_1>0$ such that the following holds.
For every fixed $0<\varepsilon\le\varepsilon_1$ there exists
$s_0=s_0(\varepsilon)$ with the following property.  Put
\[
 \vartheta=\frac{\varepsilon^2}{4},
 \qquad
 C_0=1000\varepsilon^{-6}.
\]
Let $F$ be an $s$-critical graph with $s\ge s_0$ satisfying
\[
 |F|\ge C_0s,
 \qquad
 \omega(F)\ge\lceil\varepsilon s\rceil.
\]
Then $F$ has an independent set $I$ such that $\chi(F-I)=s-1$ and
\begin{equation}\label{eq:oneStep}
 \tri_F(I)
 \ge(1+\vartheta)\binom{s-1}{2}-m_F(I),
\end{equation}
where $\tri_F(I)$ is the number of triangles meeting $I$ and $m_F(I)$
is the number of edges incident with $I$.  Since $I$ is independent,
\[
 m_F(I)=\sum_{v\in I}d_F(v)=e_F(I,V(F)\setminus I).
\]
\end{lemma}

\begin{proof}
For an independent set $U$, give $U$ weight $|U|$ if it is disjoint
from every clique $Q$ with $|Q|\ge\lceil\varepsilon s\rceil$, and
weight $(1+\varepsilon^2/10)|U|$ otherwise.  Let $I$ have maximum modified
weight.  Since $F-I$ is a proper subgraph of the $s$-critical graph
$F$, we have $\chi(F-I)\le s-1$.  Since $I$ is independent,
$\chi(F)\le1+\chi(F-I)$, and therefore $\chi(F-I)=s-1$.

The proof of Theorem~1.5 in \cite{FTZ} considers two cases.  Suppose first that $I$ is disjoint from every clique $Q$ with
$|Q|\ge\lceil\varepsilon s\rceil$.
Let $J$ be a maximum-cardinality independent set.  Its modified weight
is at least $|J|=\alpha(F)$, whereas the present case gives
$w(I)=|I|$.  Maximality of $w(I)$ therefore yields
\[
 |I|=w(I)\ge w(J)\ge\alpha(F).
\]
Since $I$ is independent, equality holds:
\begin{equation}\label{eq:modifiedCaseOneAlpha}
 |I|=\alpha(F)\ge\frac{|F|}{\chi(F)}\ge C_0.
\end{equation}
Let $Q$ be a maximum clique, so
$|Q|\ge\lceil\varepsilon s\rceil$.  We verify the consequence of
modified-weight maximality that will be used below.  If some $u\in Q$
had fewer than $(\varepsilon^2/20)|I|$ neighbors in $I$, then
\[
 I_u=(I\setminus N(u))\cup\{u\}
\]
would be independent and would meet the large clique $Q$.  Therefore
\begin{align*}
 w(I_u)
 &\ge\left(1+\frac{\varepsilon^2}{10}\right)
       \left(1-\frac{\varepsilon^2}{20}\right)|I|\\
 &>|I|=w(I),
\end{align*}
where the strict inequality holds for
$0<\varepsilon\le\varepsilon_1$ after decreasing the absolute constant
$\varepsilon_1$ if necessary.  This contradicts the choice of $I$.
Hence every $u\in Q$ has at least
$(\varepsilon^2/20)|I|$ neighbors in $I$, and consequently
\[
 \frac{e(I,Q)}{|I|}
 \ge\frac{\varepsilon^2|Q|}{20}
 \ge\frac{\varepsilon^3s}{20}.
\]
Convexity and \eqref{eq:modifiedCaseOneAlpha} now give
\[
 \tri_F(I)
 \ge\sum_{v\in I}\binom{|N(v)\cap Q|}{2}
 \ge |I|\binom{e(I,Q)/|I|}{2}.
\]
For all sufficiently large $s$, we have $e(I,Q)/|I|\ge\varepsilon^3s/20\ge1$,
so $x\mapsto\binom{x}{2}$ is increasing throughout the relevant range.
Using the two lower bounds above, the last expression is at least
\begin{align*}
 1000\varepsilon^{-6}
 \binom{\varepsilon^3s/20}{2}
 &=\frac54s^2-25\varepsilon^{-3}s.
\end{align*}
Since $2\binom{s-1}{2}=s^2-3s+2$, this is at least
$2\binom{s-1}{2}$ for fixed $\varepsilon$ and all sufficiently large
$s$.  This is stronger than \eqref{eq:oneStep}.

Suppose now that some $w\in I$ lies in a clique of order at least
$\lceil\varepsilon s\rceil$.  A maximum independent set has modified weight at
least $\alpha(F)$, while $w(I)=(1+\varepsilon^2/10)|I|$; hence
\begin{equation}\label{eq:modifiedAlphaBound}
 \alpha(F)\le(1+\varepsilon^2/10)|I|.
\end{equation}
For $v\in I\setminus\{w\}$, Tur\'an's theorem in $F[N(v)]$ and
$d(v)\ge s-1$ yield
\[
 e(F[N(v)])
 \ge\frac{(s-1)^2}{2\alpha(F)}-\frac{d(v)}2.
\]
The vertex $w$ lies in at least
$(\varepsilon^2/2)\binom s2$ triangles for all sufficiently large $s$.
Since $I$ is independent, a triangle meets at most one vertex of $I$.
Summing the preceding estimates and using
\eqref{eq:modifiedAlphaBound}, we obtain
\begin{align*}
 \tri_F(I)
 &\ge
 \frac{(|I|-1)(s-1)^2}
      {2(1+\varepsilon^2/10)|I|}
 +\frac{\varepsilon^2}{2}\binom s2
 -\frac12\sum_{v\in I\setminus\{w\}}d(v)\\
 &\ge
 \frac{(|I|-1)(s-1)^2}
      {2(1+\varepsilon^2/10)|I|}
 +\frac{\varepsilon^2}{2}\binom s2
 -m_F(I).
\end{align*}
Here the last line uses
$\frac12\sum_{v\in I\setminus\{w\}}d(v)\le m_F(I)$.
Moreover $|I|\ge C_0/(1+\varepsilon^2/10)$, and hence, after division by
$\binom s2$, the first positive term is at least
\[
 \frac{|I|-1}{|I|}\frac{s-1}{s}
 \frac1{1+\varepsilon^2/10}
 \ge1-\frac{\varepsilon^2}{5}
\]
for all sufficiently large $s$.  Adding the
$\varepsilon^2/2$ contribution gives
\[
 \tri_F(I)
 \ge(1+\varepsilon^2/4)\binom s2-m_F(I)
 \ge(1+\vartheta)\binom{s-1}{2}-m_F(I),
\]
as required.
\end{proof}

\begin{proof}[Proof of \cref{thm:uniformStability}]
It is enough to take $0<\beta\le1/2$.  Fix once and for all
$0<\varepsilon\le\min\{\varepsilon_0,\varepsilon_1\}$, with the
constants from
\cref{lem:smallCliqueSurplus,lem:modifiedWeight}.  Put
\[
 \vartheta=\frac{\varepsilon^2}{4},
 \qquad
 a=\frac{\varepsilon^{2/3}}4,
 \qquad
 C_0=1000\varepsilon^{-6}.
\]
Then $a^3=\vartheta/16$.  Let $s_0(\varepsilon)$ be the threshold in
\cref{lem:modifiedWeight}, and assume that $s$ is large enough that
$as\ge s_0(\varepsilon)$.  Thus every application of
\cref{lem:modifiedWeight} before stopping is legitimate.

Let $C>0$ be fixed and let $G$ be an $s$-critical graph with
$|G|\le Cs$ and $\omega(G)\le(1-\beta)s$.  Start with
$G_0=G_0'=G$.  At every later stage $i\ge1$, choose an induced
$(s-i)$-critical subgraph $G_i'\subseteq G_i$.  Stop at the first index
$j$ for which at least one of the following
occurs:
\begin{enumerate}[label=\textup{(\alph*)}]
\item $s-j\le as$;
\item $\omega(G_j')<\varepsilon(s-j)$;
\item $|G_j'|<C_0(s-j)$.
\end{enumerate}
Before stopping, apply \cref{lem:modifiedWeight}, choose the resulting
independent set $I_i$, and put $G_{i+1}=G_i'-I_i$.

For each stage put
\[
 m_i=m_{G_i'}(I_i)
    =e_{G_i'}(I_i,V(G_i')\setminus I_i)
    =e(G_i')-e(G_{i+1}),
\]
where the last identity uses that $I_i$ is independent.  Consequently
\begin{align*}
 \sum_{i=0}^{j-1}m_i
 &=e(G_0')-e(G_j)
   +\sum_{i=1}^{j-1}\bigl(e(G_i')-e(G_i)\bigr)\\
 &\le e(G),
\end{align*}
because $G_0'=G$ and $G_i'\subseteq G_i$.  The triangles counted at
different stages are disjoint.  Applying \eqref{eq:oneStep} with chromatic
parameter $s-i$ therefore gives the following telescoping lower bound; the
binomial sum itself telescopes exactly:
\begin{align}
 \tri(G)-\tri(G_j')
 &\ge(1+\vartheta)
   \sum_{i=0}^{j-1}\binom{s-i-1}{2}-\sum_{i=0}^{j-1}m_i\notag\\
 &\ge(1+\vartheta)
   \left(\binom s3-\binom{s-j}{3}\right)-e(G).
 \label{eq:telescoping}
\end{align}
Write $u=s-j$.  Since $|G|\le Cs$, we have $e(G)=O_C(s^2)$.

Define
\[
 \lambda_a=\frac{\vartheta-(1+\vartheta)a^3}{6},
 \qquad
 \lambda_b=\min\left\{\frac{\vartheta}{6},\sigma_0\right\},
 \qquad
 \lambda_{c1}=\frac{\vartheta}{6}
      \left(1-\left(1-\frac\beta2\right)^3\right).
\]
These constants are positive; for $\lambda_a$ this follows from
$a^3=\vartheta/16$.

If (a) occurs, then $u\le as$, and
\begin{align*}
 \tri(G)-\binom s3
 &\ge \vartheta\binom s3-(1+\vartheta)\binom u3-O_C(s^2)\\
 &\ge \lambda_a s^3-O_C(s^2).
\end{align*}

Assume now that (a) does not occur, so $u\ge as$.  If (b) occurs, then
\[
 e(G_j')\le e(G)\le\frac{C^2}{2a^2}u^2.
\]
The graph $G_j'$ is $u$-critical and has
$\omega(G_j')<\varepsilon u\le\varepsilon_0u$.  Hence
\cref{lem:smallCliqueSurplus} gives
\[
 \tri(G_j')\ge\binom u3+\sigma_0u^3.
\]
Writing $x=u/s\in[a,1]$ and using \eqref{eq:telescoping}, we obtain
\begin{align*}
 \tri(G)-\binom s3
 &\ge \vartheta\left(\binom s3-\binom u3\right)
      +\sigma_0u^3-O_C(s^2)\\
 &\ge \left(\frac\vartheta6(1-x^3)+\sigma_0x^3\right)s^3-O_C(s^2)\\
 &\ge \lambda_b s^3-O_C(s^2).
\end{align*}

It remains to consider (c), with (a) and (b) false.  If
$j\ge\beta s/2$, then condition (c) gives $|G_j'|<C_0u$.  Thus
\cref{thm:FTZlinear} applies with the fixed order constant $C_0$ and gives
$\tri(G_j')\ge\binom u3$.  Since $u\le(1-\beta/2)s$,
\[
 \tri(G)-\binom s3
 \ge\vartheta\left(\binom s3-\binom u3\right)-O_C(s^2)
 \ge\lambda_{c1}s^3-O_C(s^2).
\]

Finally suppose $j<\beta s/2$.  Then $u>(1-\beta/2)s$ and
\[
 \omega(G_j')\le\omega(G)\le(1-\beta)s
 \le(1-\beta/2)u.
\]
By (c), $|G_j'|<C_0u$.  Apply
\cref{prop:fixedStability} with $C_0$ and $\beta/2$ to obtain
\[
 \tri(G_j')\ge\binom u3+\gamma_0u^3,
\]
where $\gamma_0>0$ depends only on $\beta$ and the fixed choice of
$\varepsilon$.  Put
\[
 \lambda_{c2}=\gamma_0\left(1-\frac\beta2\right)^3>0.
\]
Inserting this estimate into \eqref{eq:telescoping} gives
\[
 \tri(G)-\binom s3
 \ge\vartheta\left(\binom s3-\binom u3\right)
      +\gamma_0u^3-O_C(s^2)
 \ge\lambda_{c2}s^3-O_C(s^2).
\]

Set
\[
 \gamma(\beta)=\frac12
 \min\{\lambda_a,\lambda_b,\lambda_{c1},\lambda_{c2}\}>0.
\]
The four estimates above show that
$\tri(G)\ge\binom s3+\gamma(\beta)s^3$ once
$s\ge s_0(C,\beta)$ is large enough to absorb the terms $O_C(s^2)$.
Thus the lower threshold may depend on $C$, while the surplus does not.
\end{proof}

\section{Assembly of the proof}\label{sec:assembly}

\begin{proof}[Proof of \cref{thm:main}]
Let $c_0,\delta_0$ be supplied by \cref{thm:nearClique}.  Put
\[
 \beta=\delta_0/4
\]
and let $\gamma=\gamma(\beta)$ be supplied by \cref{thm:uniformStability}.  Choose $\eta>0$ sufficiently small that
\begin{equation}\label{eq:etaChoices}
 \eta\le\delta_0/4
 \qquad\text{and}\qquad
 \binom{\lfloor(1-\eta)r\rfloor}{3}
 +\gamma\lfloor(1-\eta)r\rfloor^3
 \ge\binom r3
\end{equation}
for all sufficiently large $r$.  This is possible because the deficit in the binomial term is $O(\eta r^3)$, whereas $\gamma$ is independent of the order constant in the sparse core.

Let $c_\eta,L_\eta$ be supplied by \cref{thm:core} and put
\[
 c=\min\{c_0,c_\eta\}.
\]
Suppose, for a contradiction, that $G$ satisfies
\[
 \chi(G)\ge r,
 \qquad
 e(G)\le cr^3\log^2r,
 \qquad
 \tri(G)<\binom r3.
\]
By \cref{thm:core}, $G$ has an induced subgraph $H$ with
\[
 s:=\chi(H)\ge(1-\eta)r,
 \qquad
 |H|\le L_\eta r.
\]
Choose an induced $s$-critical subgraph $J\subseteq H$.  Then
\[
 |J|\le\frac{L_\eta}{1-\eta}s.
\]

If $\omega(J)\le(1-\beta)s$, apply \cref{thm:uniformStability}.
The function $x\mapsto\binom{x}{3}+\gamma x^3$ is increasing on integers
$x\ge3$, and $s\ge(1-\eta)r$ implies
$s\ge\lfloor(1-\eta)r\rfloor$.  Hence \eqref{eq:etaChoices} gives
\[
 \tri(G)\ge\tri(J)
 \ge\binom s3+\gamma s^3
 \ge\binom r3,
\]
a contradiction.

Otherwise
\[
 \omega(G)\ge\omega(J)>(1-\beta)s
 \ge(1-\beta)(1-\eta)r
 \ge(1-\delta_0)r.
\]
Since $e(G)\le c_0r^3\log^2r$, \cref{thm:nearClique} again gives $\tri(G)\ge\binom r3$, a contradiction.  This proves the theorem for all sufficiently large $r$.

For the finitely many integers $3\le r<r_0$, decrease $c$ once more.
For each such $r$ for which a counterexample exists, let
\[
 M_r=\min\left\{e(H):\chi(H)\ge r,
                  \ \tri(H)<\binom r3\right\}.
\]
Then $M_r$ is a positive integer.  Put
\[
 c_{\mathrm{fin}}
 =\frac12\min_{\substack{3\le r<r_0\\ M_r\text{ defined}}}
   \frac{M_r}{r^3\log^2r},
\]
with any positive value if the indexing set is empty, and replace $c$ by
$\min\{c,c_{\mathrm{fin}}\}$.  Now
$e(G)\le cr^3\log^2r<M_r$ for every remaining $r$, so equality at the
finite minimum cannot occur.  This completes the proof for every
$r\ge3$.
\end{proof}


\begin{thebibliography}{99}

\bibitem{FTZ}
J.~Fox, J.~Tidor, and S.~Zhang,
\newblock Triangle Ramsey numbers of complete graphs,
\newblock \emph{J. Combin. Theory Ser. B} \textbf{176} (2026), 268--286.
\newblock \href{https://doi.org/10.1016/j.jctb.2025.08.004}{doi:10.1016/j.jctb.2025.08.004}.

\bibitem{ErdosSos}
P.~Erd\H{o}s and V.~T. S\'os,
\newblock Some remarks on Ramsey's and Tur\'an's theorem,
\newblock in \emph{Combinatorial Theory and its Applications I--III},
North-Holland, Amsterdam, 1970, pp.~395--404.

\bibitem{Harris}
D.~G. Harris,
\newblock Some results on chromatic number as a function of triangle count,
\newblock \emph{SIAM J. Discrete Math.} \textbf{33} (2019), 546--563.
\newblock \href{https://doi.org/10.1137/17M115918X}{doi:10.1137/17M115918X}.


\bibitem{Kim}
J.~H. Kim,
\newblock The Ramsey number $R(3,t)$ has order of magnitude $t^2/\log t$,
\newblock \emph{Random Structures Algorithms} \textbf{7} (1995), 173--207.
\newblock \href{https://doi.org/10.1002/rsa.3240070302}{doi:10.1002/rsa.3240070302}.

\bibitem{Nilli}
A.~Nilli,
\newblock Triangle-free graphs with large chromatic numbers,
\newblock \emph{Discrete Math.} \textbf{211} (2000), 261--262.
\newblock \href{https://doi.org/10.1016/S0012-365X(99)00109-0}{doi:10.1016/S0012-365X(99)00109-0}.

\bibitem{PoljakTuza}
S.~Poljak and Z.~Tuza,
\newblock Bipartite subgraphs of triangle-free graphs,
\newblock \emph{SIAM J. Discrete Math.} \textbf{7} (1994), 307--313.
\newblock \href{https://doi.org/10.1137/S0895480191196824}{doi:10.1137/S0895480191196824}.

\bibitem{Vu}
V.~H. Vu,
\newblock A general upper bound on the list chromatic number of locally sparse graphs,
\newblock \emph{Combin. Probab. Comput.} \textbf{11} (2002), 103--111.
\newblock \href{https://doi.org/10.1017/S0963548301004898}{doi:10.1017/S0963548301004898}.

\end{thebibliography}
\end{document}